\documentclass[a4paper,10pt]{article}
\usepackage{amssymb,amsmath,amsthm,amsfonts}
\usepackage{bookmark}
\usepackage{mathrsfs}
\usepackage{multirow}
\usepackage{graphicx}
\usepackage{authblk}
\usepackage{indentfirst}
\usepackage{multicol}
\usepackage{tabu}
\usepackage{url}
\usepackage{fancyhdr}
\usepackage[numbers]{natbib}
\usepackage[all]{xy}
\usepackage{mathtools}
\usepackage[english]{babel}
\usepackage{colortbl}
\usepackage{caption}
\usepackage{hyperref}\usepackage{subfigure}
\usepackage{tikz}\usepackage{float}\usepackage{CJK}
\usepackage{xcolor}\usepackage{mdframed}\usepackage{footmisc}\usepackage{pgfplots}

\newtheorem{theorem}{Theorem}[section]
\newtheorem{proposition}[theorem]{Proposition}
\newtheorem{lemma}[theorem]{Lemma}
\newtheorem{corollary}[theorem]{Corollary}

\newtheorem{definition}{Definition}[section]
\newtheorem{example}{Example}[section]

\newcommand{\im}{{\mathrm{im}\hspace{0.1em}}}

\newcommand{\rank}{{\mathrm{rank}\hspace{0.1em}}}

\definecolor{blue}{rgb}{0.0, 0.5, 1.0}

\definecolor{red}{rgb}{1.0, 0.0, 0.0}

\definecolor{green}{rgb}{0.00,0.60,0}

\usepackage{xr}
\makeatletter
\newcommand*{\addFileDependency}[1]{
	\typeout{(#1)}
	\@addtofilelist{#1}
	\IfFileExists{#1}{}{\typeout{No file #1.}}
}
\makeatother

\title{Persistent Mayer homology and persistent Mayer Laplacian}

\author[1]{Li Shen}
\author[2,1]{Jian Liu}
\author[1,3,4]{Guo-Wei Wei \thanks{Corresponding author: weig@msu.edu}}
\affil[1]{Department of Mathematics, Michigan State University, MI 48824, USA}
\affil[2]{Mathematical Science Research Center, Chongqing University of Technology, Chongqing 400054, China}
\affil[3]{Department of Electrical and Computer Engineering, Michigan State University, MI 48824, USA}
\affil[4]{Department of Biochemistry and Molecular Biology, Michigan State University, MI 48824, USA}

\makeatletter
\renewcommand*{\@fnsymbol}[1]{\ensuremath{\ifcase#1\or \dagger\or *\or *\or
		\mathsection\or \else\@ctrerr\fi}}
\makeatother
\date{}

\begin{document}
	\begin{CJK*}{GBK}{kai}
		\CJKtilde
		\maketitle
		
		\paragraph{Abstract}
		In algebraic topology, the differential (i.e., boundary operator) typically satisfies $d^{2}=0$. However,  the generalized differential $d^{N}=0$ for an integer $N\geq 2$ has been studied in terms of Mayer homology on $N$-chain complexes for more than eighty years. We introduce Mayer Laplacians on $N$-chain complexes. We show that both Mayer homology and Mayer Laplacians offer considerable application potential, providing topological and geometric insights to spaces.  We also introduce persistent Mayer homology and persistent Mayer Laplacians at various $N$. The Wasserstein distance and stability of persistence diagrams associated with Mayer homology are investigated. Our computational experiments indicate that the topological features offered by  persistent Mayer homology and spectrum given by persistent Mayer Laplacians hold substantial promise for large, complex, and diverse data. We envision that the present work serves as an inaugural step towards integrating Mayer homology and Mayer Laplacians into the realm of topological data analysis.



	     \newpage
		 {\setcounter{tocdepth}{4} \tableofcontents}
	     \newpage
		
\section{Introduction}\label{section:introduction}

Topological data analysis (TDA) stands at the forefront of innovative methodologies in the field of data science, employing tools derived from  algebraic topology and differential geometry to analyze the topological invariants and geometry shapes of  complex datasets. In contrast to conventional approaches that often concentrate on numerical attributes of data, TDA distinguishes itself by prioritizing the extraction of significant topological invariants and geometrical shapes. These features play a pivotal role in capturing the nuanced patterns and relationships embedded within the data. The power of TDA is exemplified in the topological deep learning
paradigm \cite{cang2017topologynet}.

A particularly noteworthy aspect of TDA is persistent homology, a concept that extends the utility of traditional topological techniques \cite{carlsson2009computing}. As a key tool of TDA, persistent homology enables the identification and preservation of topological features across various scales within the dataset. Unlike methods of static data analyses, persistent homology captures the evolution of features over different scales, providing a dynamic and comprehensive understanding of the underlying topological structures. By incorporating this filtration of data, persistent homology, excels at discerning persistent patterns and revealing the enduring topological signatures that may be overlooked by traditional methods. This synergy enhances the robustness and depth of insights gained from TDA, making it an invaluable approach for unraveling complex relationships within diverse datasets.

As early as 1942, Walther Mayer introduced a novel homology theory that was not based on chain complexes, but rather on a structure known as an $N$-chain complex \cite{mayer1942new}. This $N$-chain complex can be understood simply as having a boundary operator $d$ that satisfies $d^N = 0$, rather than the typical  $d^2 = 0$ in persistent homology. This kind of structure appears to be more intriguing and facilitates the Mayer homology (sometimes called generalized homology). In \cite{spanier1949mayer}, Mayer considered $N$-chain complexes with coefficients in the field of integers modulo $p$. He provided a correspondence between Mayer homology of simplicial complexes and simplicial homology. This demonstrates that Mayer homology and simplicial homology can be mutually derived from each other. Reviewing the traditional differentials in simplicial complexes, Mayer homology and simplicial homology typically involve linear combinations of face operators with coefficients of +1 or -1. Drawing inspiration from this concept, we extend these coefficients to be the $N$-th primitive roots of unity. This adjustment results in an $N$-chain complex and its corresponding algebraic theory.
Let $q$ represents the primitive $N$-th root of unity. Utilizing $q$ enables the construction of an $N$-differential. This construction gives rise to the derivation of a $q$-analog for a differential graded algebra, subsequently allowing for the computation of Tor- and Ext-groups \cite{dubois1996generalized,kassel1998algebre}. the tensor product structure on the $q$-differential graded algebra has been explored, as discussed in \cite{sitarz1998tensor}. It is worth noting that there are multiple ways to construct an $N$-chain complex from a simplicial complex in the literature \cite{dubois1998d}. In \cite{abramov2006graded}, the author applied the $q$-differential to a reduced quantum plane and studied the corresponding exterior calculus on the reduced quantum. Recently, other research on $N$-chain complexes has been proposed \cite{lu2020gorenstein,lu2020cartan}.

Inspired by $N$-chain complexes and Mayer homology, we believe that Mayer homology can reveal additional topological and geometric features of a space, which is highly beneficial for the TDA of large, complex, and diverse data. For a simplicial complex, where all simplices serve as blocks and blocks of different dimensions form its algebraic and geometric structure, the standard chain complex provides a boundary operator or differential that describes the connections between simplices of adjacent dimensions.
In contrast, for a general $N$-chain complex, the $N$-differential and its composition establish connections between simplices of different dimensions. This feature is absent in traditional chain complexes where the composition of differentials results in zero. In this sense, $N$-chain complexes better capture profound relationships between simplices of varying dimensions. Consequently, Mayer homology and the corresponding Mayer Laplacian can more effectively characterize these relationships among simplices of different dimensions.

Persistent homology theory is the main workhorse in TDA and has seen substantial enrichment and development in recent years. From the standpoint of persistent parameters, researchers have delved beyond single persistent homology, exploring multi-persistent homology \cite{carlsson2007theory,carlsson2009computing}, Zig-zag persistent homology \cite{carlsson2010zigzag}, Cayley persistent homology \cite{bi2022cayley}, and many other variants.
However, persistent homology has many limitations, including its inability to capture geometric and topological features beyond topological invariants. Wei and his coworkers introduced persistent Laplacians on smooth manifolds \cite{chen2021evolutionary} and point clouds \cite{wang2020persistent}  to address the limitations of persistent homology. The harmonic part of the spectrum of the persistent Laplacian operator corresponds to persistent homology information, while the non-harmonic part provides geometric insights into the simplicial complex. Persistent Laplacians show superior performance over persistent homology, leading to successful forecasting emerging dominant viral variants \cite{chen2022persistent}.  Both persistent homology and persistent Laplacian can be defined on many topological objects beyond simplicial complex, resulting in persistent hypergraph homology/Laplacian \cite{bressan2016embedded,liu2021persistent}, persistent sheaf Laplacian \cite{wei2021persistent}, persistent path Laplacian \cite{wang2023persistent}, and persistent hyperdigraph homology/Laplacian \cite{chen2023persistent}.

However, it is worth noting that all the aforementioned formulations are built upon the construction of chain complexes. The $N$-chain complex exhibits characteristics distinct from those of usual chain complexes. In this work, we introduce  persistent Mayer homology and persistent Mayer Laplacians on $N$-chain complexes.
It is worth noting that Mayer homology may not necessarily gives rise to a homotopy invariant. Specifically, for two simplicial complexes that are homotopy equivalent, their Mayer homology may not be isomorphic. This implies that persistent Mayer homology and persistent Mayer Laplacians can reflect certain geometric structures and more topological features of simplicial complexes.
Our computations indicate that persistent Mayer homology often provides a wealth of multiscale information, comparable in many instances to the information obtained by combining the usual persistent homology as well as its associated persistent Laplacians. This underscores the strong capability of persistent Mayer homology in characterizing both geometric and topological features. Furthermore, the computation of persistent Mayer homology is significantly faster than computing the usual persistent Laplacian, highlighting a distinct advantage of persistent Mayer homology.

In this work, we employ the $N$-chain complex and Mayer homology to construct a generalized version of persistent homology theory based on $N$-differentials. Specifically, by considering the multiscale information from datasets, we introduce persistent Mayer homology (PMH) and persistent Mayer Laplacians (PMLs). We investigate the Wasserstein  distance and the stability of the persistence diagram corresponding to PMH. For a given simplicial complex, an $N$-chain complex can be constructed over the complex number field $\mathbb{C}$, and the $N$-differential on this $N$-chain complex is determined by $N$-th primitive roots of unity. This aligns with the conventional notion of differentials on chain complexes, where differentials are linear combinations of face operators with coefficients +1 and -1 (quadratic roots of unity). It is worth noting that the Mayer Laplacian can be precisely formulated as a well-behaved construction on the complex number field $\mathbb{C}$ with the Hermitian adjoint.
For a given persistence parameter, PMH and PMLs provide a family of topological features ($q=1,2,\dots,N-1$). By computing  examples on real molecules, we   observe  that these features exhibit richer topological and geometric information compared to the usual persistent simplicial homology. Computations and examples are presented to elucidate the characteristics of PMH and PMLs.

The paper is organized as follows. In Section \ref{section:Mayer_homology}, we review the $N$-chain complex and Mayer homology to establish notations. In Section \ref{section:persistence_Mayer}, we introduce persistent Mayer homology and persistent Mayer Laplacians for simplicial complexes. Section \ref{section:applications} illustrates the applications of the proposed persistent Mayer homology and persistent Mayer Laplacians with two molecules. Finally, Section \ref{section:conclusion} provides a summary of our work and discusses potential future directions.

\section{$N$-chain complex and Mayer homology}\label{section:Mayer_homology}

In this section, we  review fundamental concepts, including the $N$-chain complex and Mayer homology. Moreover, for a given simplicial complex, it is possible to construct multiple $N$-chain complexes. We   concentrate on a specific construction, which will be applied to our examples and dataset later on. Additionally, we   introduce Laplacian operators  on $N$-chain complexes. This section encompasses some properties of $N$-chain complexes and Mayer homology, along with examples of related computations.
From now on, the ground field is assumed to be the field $\mathbb{K}$. The $N$-chain complex and Mayer homology can be also built on a commutative ring with unit.

\subsection{Mayer homology}
From now on, $N$ is always an integer $\geq 2$.
\begin{definition}
An \emph{$N$-chain complex} consists of a graded $\mathbb{K}$-linear space $C_{\ast}=(C_{n})_{n\geq 0}$, equipped with a linear map $d: C_{\ast} \to C_{\ast-1}$ of degree $-1$ satisfying $d^{N}=0$.
The linear map $d_{\ast}:C_{\ast}\to C_{\ast-1}$ is called the \emph{$N$-differential ($N$-boundary operator)}.
\end{definition}
The following diagram illustrates the $N$-differential within the $N$-chain complex. Each horizontal sequence represents a chain complex corresponding to  stage $q$. The vertical sequences are given by the identity  map (id) or by the $N$-differential $d$.
\begin{equation*}
  \xymatrix{
  \cdots\ar@{->}[r]^{d}&C_{n+N-1}\ar@{->}[r]^{d^{N-1}}\ar@{->}[d]^{d}&C_{n}\ar@{->}[r]^{d}\ar@{->}[d]^{\mathrm{id}}&C_{n-1}\ar@{->}[r]^{d^{N-1}}\ar@{->}[d]^{d}&C_{n-N}\ar@{->}[r]^{d}\ar@{->}[d]^{\mathrm{id}}&C_{n-N-1}\ar@{->}[r]^{d^{N-1}}\ar@{->}[d]^{d}&\cdots\\
  \cdots\ar@{->}[r]^{d^{2}}&C_{n+N-2}\ar@{->}[r]^{d^{N-2}}\ar@{->}[d]^{d}&C_{n}\ar@{->}[r]^{d^{2}}\ar@{->}[d]^{\mathrm{id}}&C_{n-2}\ar@{->}[r]^{d^{N-2}}\ar@{->}[d]^{d}&C_{n-N}\ar@{->}[d]^{\mathrm{id}}\ar@{->}[r]^{d}&C_{n-N-2}\ar@{->}[r]^{d^{N-2}}\ar@{->}[d]^{d}&\cdots \\
  \cdots&\vdots\ar@{->}[d]^{d}&\vdots\ar@{->}[d]^{\mathrm{id}}&\vdots\ar@{->}[d]^{d}&\vdots\ar@{->}[d]^{\mathrm{id}}&\vdots\ar@{->}[d]^{d}&\cdots\\
  \cdots\ar@{->}[r]^{d^{q}}&C_{n+N-q}\ar@{->}[r]^{d^{N-q}}\ar@{->}[d]^{d}&C_{n}\ar@{->}[r]^{d^{q}}\ar@{->}[d]^{\mathrm{id}}&C_{n-q}\ar@{->}[r]^{d^{N-q}}\ar@{->}[d]^{d}&C_{n-N}\ar@{->}[r]^{d^{q}}\ar@{->}[d]^{\mathrm{id}}&C_{n-N-q}\ar@{->}[r]^{d^{N-q}}\ar@{->}[d]^{d}&\cdots\\
  \cdots&\vdots\ar@{->}[d]^{d}&\vdots\ar@{->}[d]^{\mathrm{id}}&\vdots\ar@{->}[d]^{d}&\vdots\ar@{->}[d]^{\mathrm{id}}&\vdots\ar@{->}[d]^{d}&\cdots\\
  \cdots\ar@{->}[r]^{d^{N-2}}&C_{n+2}\ar@{->}[r]^{d^{2}}\ar@{->}[d]^{d}&C_{n}\ar@{->}[r]^{d^{N-2}}\ar@{->}[d]^{\mathrm{id}}&C_{n-N+2}\ar@{->}[r]^{d^{2}}\ar@{->}[d]^{d}&C_{n-N}\ar@{->}[r]^{d^{N-2}}\ar@{->}[d]^{\mathrm{id}}&C_{n-2N+2}\ar@{->}[r]^{d^{2}}\ar@{->}[d]^{d}&\cdots\\
  \cdots\ar@{->}[r]^{d^{N-1}}&C_{n+1}\ar@{->}[r]^{d}&C_{n}\ar@{->}[r]^{d^{N-1}}&C_{n-N+1}\ar@{->}[r]^{d}&C_{n-N}\ar@{->}[r]^{d^{N-1}}&C_{n-2N+1}\ar@{->}[r]^{d}&\cdots
  }
\end{equation*}
In particular, when $N=2$, the $N$-chain complex reduces to the usual chain complex.
\begin{definition}
A \emph{morphism $f:(C_{\ast},d)\to (C_{\ast}',d')$ of $N$-chain complexes} is a linear map of degree zero such that $f\circ d=d'\circ f$.
 \end{definition}

Let $(C_{\ast},d)$ be an $N$-chain complex. For each $1\leq q\leq N-1$, the \emph{space of the $q$-th $n$-cycles} is defined by $Z_{n,q}=\{x\in C_{n}|d^{q}x=0\}$.
The \emph{space of the $q$-th $n$-boundaries} is given by $B_{n,q}=\{d^{N-q}x|x\in C_{N+p-n}\}$.
It follows that $B_{n,q}\subseteq Z_{n,q}$. Let us denote $d_{n}:C_{n}\to C_{n-1}$. In particular, for $N=3$, we can prove that $d_{n}C_{n}\subseteq B_{n-1,2}$, $d_{n}Z_{n,2}\subseteq Z_{n-1,1}\cap B_{n-1,2}$, $d_{n}Z_{n,1}=0$, and $d_{n}B_{n,2}\subseteq B_{n-1,1}$.
\begin{figure}[H]
  \centering
  \includegraphics[width=0.6\textwidth]{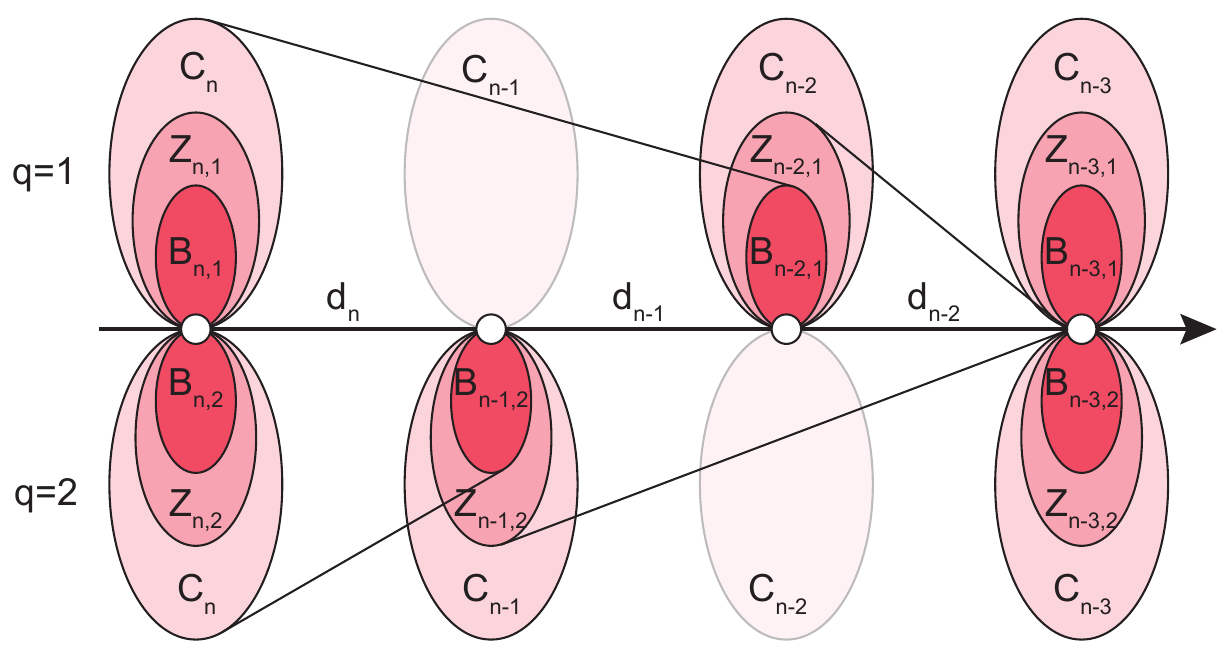}\\
  \caption{Illustration of the boundary operators and chain, cycle, and boundary groups of the $N$-chain complex.}\label{figure:illustration}
\end{figure}
The \emph{Mayer homology} of the $N$-chain complex $(C_{\ast},d)$ is defined as
\begin{equation}
  H_{n,q}(C_{\ast},d):=Z_{n,q}/B_{n,q},\quad n\geq 0.
\end{equation}
The rank of $H_{n,q}(C_{\ast},d)$ is defined as the \emph{Mayer Betti number} of the $N$-chain complex $(C_{\ast},d)$.
The idea of Mayer homology was first introduced by Mayer in 1942 \cite{mayer1942new}. In Mayer's paper, he constructed the $N$-chain complex on simplicial complexes over the field $\mathbb{Z}/p$. Here, $p$ is a prime number. And the name of Mayer homology first appeared in \cite{spanier1949mayer}, which showed the relationship between Mayer homology and the classical homology of simplicial complexes.

\begin{example}
Consider the graded vector space $\mathbb{Z}_{3}[x]$, with the grading $(\mathbb{Z}_{3}[x])_{n}=\mathbb{Z}_{3}x^{n}$ and the basis $1,x,x^{2},\dots, x^{k},\dots$. Here, $\mathbb{Z}_{3}$ is the field with elements $0,1,2$ modulo $3$. Consider the linear map $d:\mathbb{Z}_{3}[x]\to \mathbb{Z}_{3}[x]$ given by $dx^{n}=nx^{n-1}$ and $d(1)=0$. It follows that
\begin{equation*}
  d^{3}x=n(n-1)(n-2)x^{n-3},n\geq 3
\end{equation*}
or $d^{3}x=0$ for $0\leq n\leq 2$. Since the number 3 is a factor of $n(n-1)(n-2)$, we have $n(n-1)(n-2) \equiv 0$ modulo 3. This indicates that $d^{3}=0$. By a straightforward calculation, we have
\begin{equation*}
\begin{split}
  Z_{n,1} = &\left\{
              \begin{array}{ll}
                \mathbb{Z}_{3}x^{n}, & \hbox{$n=3k,k\in \mathbb{Z}_{\geq 0}$;} \\
                0, & \hbox{\rm otherwise.}
              \end{array}
            \right.\\
  Z_{n,2} = & \left\{\begin{array}{ll}
                \mathbb{Z}_{3}x^{n}, & \hbox{$n=3k,3k+1,k\in \mathbb{Z}_{\geq 0}$;} \\
                0, & \hbox{\rm otherwise.}
              \end{array}
            \right.\\
  B_{n,1} = &\left\{\begin{array}{ll}
                \mathbb{Z}_{3}x^{n}, & \hbox{$n=3k,k\in \mathbb{Z}_{\geq 0}$;} \\
                0, & \hbox{\rm otherwise.}
              \end{array}
            \right.\\
  B_{n,2} = &\left\{\begin{array}{ll}
                \mathbb{Z}_{3}x^{n}, & \hbox{$n=3k,3k+1,k\in \mathbb{Z}_{\geq 0}$;} \\
                0, & \hbox{\rm otherwise.}
              \end{array}
            \right.
\end{split}
\end{equation*}
By definition, the Mayer homology is given by
\begin{equation*}
  H_{n,1}(\mathbb{Z}_{3}[x])=H_{n,2}(\mathbb{Z}_{3}[x])=0, \quad n\geq 0.
\end{equation*}
Now, let $A_{m}=\mathbb{Z}_{3}\{1,x,\dots,x^{3m+1}\}$ be the graded vector space generated by $1,x,\dots,x^{3m+1}$. One has
\begin{equation*}
\begin{split}
  Z_{n,1} = &\left\{
              \begin{array}{ll}
                \mathbb{Z}_{3}x^{n}, & \hbox{$n=3k,k=0,1,\dots,m$;} \\
                0, & \hbox{\rm otherwise.}
              \end{array}
            \right.\\
  Z_{n,2} = & \left\{\begin{array}{ll}
                \mathbb{Z}_{3}x^{n}, & \hbox{$n=3k,3k+1,k=0,1,\dots,m$;} \\
                0, & \hbox{\rm otherwise.}
              \end{array}
            \right.\\
  B_{n,1} = &\left\{\begin{array}{ll}
                \mathbb{Z}_{3}x^{n}, & \hbox{$n=3k,k=0,1,\dots,m-1$;} \\
                0, & \hbox{\rm otherwise.}
              \end{array}
            \right.\\
  B_{n,2} = &\left\{\begin{array}{ll}
                \mathbb{Z}_{3}x^{n}, & \hbox{$n=3k,3k+1,k=0,1,\dots,m-1$;} \\
                \mathbb{Z}_{3}x^{n}, & \hbox{$n=3m$;} \\
                0, & \hbox{\rm otherwise.}
              \end{array}
            \right.
\end{split}
\end{equation*}
It follows that
\begin{equation*}
  H_{n,1}(A_{m})=\left\{
                               \begin{array}{ll}
                                 \mathbb{Z}_{3}x^{n}, & \hbox{$n=3m$;} \\
                                 0, & \hbox{\rm otherwise.}
                               \end{array}
                             \right.
\end{equation*}
and
\begin{equation*}
  H_{n,2}(A_{m})=\left\{
                               \begin{array}{ll}
                                 \mathbb{Z}_{3}x^{n}, & \hbox{$n=3m+1$;} \\
                                 0, & \hbox{\rm otherwise.}
                               \end{array}
                             \right.
\end{equation*}
\end{example}

Let $f:(C_{\ast},d)\to (C_{\ast}',d')$ be a morphism of $N$-chain complexes. Since $f$ commutes with the $N$-differential, it induces the morphism of Mayer homology
\begin{equation}
  f_{\ast,q}:H_{\ast,q}(C_{\ast},d)\to H_{\ast,q}(C_{\ast}',d'),\quad [z]\mapsto [f(z)]
\end{equation}
for any $1\leq q\leq N-1$. Moreover, one has
\begin{proposition}\emph{(\cite[Proposition 1]{dubois1998d})}
If $f_{\ast,1}:H_{\ast,1}(C_{\ast},d)\to H_{\ast,1}(C_{\ast}',d')$ and $f_{\ast,N-1}:H_{\ast,N-1}(C_{\ast},d)\to H_{\ast,N-1}(C_{\ast}',d')$ are isomorphisms, then $f_{\ast,q}:H_{\ast,q}(C_{\ast},d)\to H_{\ast,q}(C_{\ast}',d')$ is an isomorphism for any $1\leq q\leq N-1$.
\end{proposition}
The above proposition   shows that if $f_{\ast,q}:H_{\ast,1}(C_{\ast},d)\to H_{\ast,1}(C_{\ast}',d')$ is an isomorphism for $q=1,N-1$, then it is an isomorphism for any $1\leq q\leq N-1$.
There are various distinctive properties associated with Mayer homology. For instance, it has been demonstrated in \cite{dubois1998d} that there exists an isomorphism of linear spaces, $H_{\ast,q}(C_{\ast},d)\cong H_{\ast,N-q}(C_{\ast},d)$. However, it does not have to be $H_{n,q}(C_{\ast},d)\cong H_{n,N-q}(C_{\ast},d)$ for a given $n$.

Let $\mathbf{Nchain}$ be the category of $N$-chain complexes, whose objects are the $N$-chain complexes, and whose morphisms are the morphisms of $N$-chain complexes. Let $\mathbf{Vec}_{\mathbb{K}}$ be the category of vector spaces over $\mathbb{K}$. Then we have the following proposition.
\begin{proposition}\label{proposition:functor1}
The Mayer homology $H_{\ast,q}:\mathbf{Nchain}\to \mathbf{Vec}_{\mathbb{K}}$ is a functor for $1\leq q\leq N-1$.
\end{proposition}
\begin{proof}
For morphisms  $f:(C_{\ast},d)\to (C_{\ast}',d')$ and $g:(C_{\ast}',d')\to (C_{\ast}'',d'')$ of $N$-chain complexes, one has
\begin{equation}
  g_{\ast,q}f_{\ast,q}([z])=g_{\ast,q}([f(z)])=[gf(z)]=(g\circ f)_{\ast,q}([z]).
\end{equation}
Here, $z\in H_{\ast,q}(C_{\ast},d)$. The left can be verified step by step.
\end{proof}
It is worth noting that the functorial property of Mayer homology is crucial for us to develop the persistence for Mayer homology. More specifically, morphisms at the $N$-chain level can always induce morphisms at the homology level. Indeed, we also require  the functorial property that maps the morphisms at the simplicial complex level to morphisms at the $N$-chain level.

The $N$-chain complex is a kind of generalization of the usual chain complex by changing the boundary operator by an $N$-boundary operator. Other than the homology of $N$-chain complexes, the homotopy for $N$-chain complexes can be also built. More precisely, two morphisms $f,g:(C_{\ast},d)\to (C_{\ast}',d')$ of $N$-chain complexes are \emph{homotopic} if there exist linear maps $h_{k}:(C_{\ast},d)\to (C_{\ast+1}',d')$ of degree 1 for $0\leq k\leq N-1$ such that $f-g=\sum\limits_{k=0}^{N-1}h_{k}d^{k}$. If $f,g:(C_{\ast},d)\to (C_{\ast}',d')$ are $N$-chain homotopic, then they induce the same morphism of Mayer homology, i.e., $f_{\ast,q}=g_{\ast,q}$ for $1\leq q\leq N-1$.

\subsection{$N$-chain complex on simplicial complexes}

From now on, for the sake of simplicity, we will always consider the case where $N$ is a prime number, and the field $\mathbb{K}$ is taken to be the complex number field $\mathbb{C}$.
Let $\xi=e^{2\pi \sqrt{-1}/N}$ be the primitive $N$-th root of unity. It follows that $\sum\limits_{i=0}^{N-1}\xi^{i}=0$. Moreover, $\sum\limits_{i=0}^{k}\xi^{i}\neq 0$ for any $0\leq k\leq N-2$.\\

Let $K$ be a simplicial complex. Let $C_{n}(K;\mathbb{C})$ be the linear space generated by the $n$-simplices of $K$ over $\mathbb{C}$. Consider the linear map $d_{n}:C_{n}(K;\mathbb{C})\to C_{n-1}(K;\mathbb{C})$ given by
\begin{equation}
  d_{n}\langle v_{0},v_{1},\dots,v_{n}\rangle=\sum\limits_{i=0}^{n}\xi^{i}\langle v_{0},\dots,\hat{v_{i}},\dots,v_{n}\rangle,\quad n\geq 1
\end{equation}
and $d_{0}=0$. Then $d:C_{\ast}(K;\mathbb{C})\to C_{\ast}(K;\mathbb{C})$ we have a linear map of degree -1. Moreover, we have
\begin{lemma}
$d^{N}=0$.
\end{lemma}
\begin{proof}
Let $\partial_{i}:K_{n}\to K_{n-1},\langle v_{0},v_{1},\dots,v_{n}\rangle\mapsto \langle v_{0},\dots,\hat{v_{i}},\dots,v_{n}\rangle$ denote the $i$-th face map of simplicial complex $K$. If $n<N$, we have $d^{N}=0$. For $r\leq n$, by induction, we can prove
\begin{equation}
  d^{r}=\left(\prod\limits_{k=1}^{r}(1+\xi+\cdots+\xi^{k-1})\right)\sum\limits_{j_{1}<\cdots<j_{r}}\xi^{j_{1}+\cdots+j_{r}-\frac{r(r-1)}{2}}\partial_{j_{1}}\cdots\partial_{j_{r}}.
\end{equation}
Note that $1+\xi+\cdots+\xi^{N-1}=0$. It follows that $d^{N}=0$.
\end{proof}
Then the construction $(C_{\ast}(K;\mathbb{C}),d)$ is an $N$-chain complex. There are various ways to construct $N$-chain complexes on a simplicial complex, and these different constructions lead to different Mayer homology \cite{dubois1998d}. In this work, we will study the $N$-chain complex constructed above. The $N$-chain complex $(C_{\ast}(K;\mathbb{C}),d)$ is over the field $\mathbb{C}$, which is more computationally feasible. In addition, we can consider the inner product structure on the $N$-chain complex $(C_{\ast}(K;\mathbb{C}),d)$, which leads to the Laplacians on the $N$-chain complex.

For $1\leq q\leq N-1$, the \emph{Mayer homology of the simplicial complex} $K$ is defined by
\begin{equation}
  H_{n,q}(K;\mathbb{C}):=H_{n,q}(C_{\ast}(K;\mathbb{C}),d),\quad n\geq 0.
\end{equation}
The Betti numbers corresponding to the Mayer homology are called the \emph{Mayer Betti numbers} of simplicial complex, denoted by $\beta_{n,q}$.

\begin{proposition}\label{proposition:functor2}
The construction $C_{\ast}(-;\mathbb{C}):\mathbf{Cpx}\to \mathbf{Nchain}$ is a functor from the category of simplicial complexes to the category of $N$-chain.
\end{proposition}
\begin{proof}
Let $\phi:K\to L$ be a morphism of simplicial complexes. The induced morphism $$C_{\ast}(\phi):(C_{\ast}(K;\mathbb{C}),d_{K})\to (C_{\ast}(L;\mathbb{C}),d_{L})$$ of $N$-chain complexes is given by $C_{\ast}(\phi)(\sigma)=\phi(\sigma)$. Indeed, for any $\sigma=\langle v_{0},v_{1},\dots,v_{n}\rangle$, we have
\begin{equation}
\begin{split}
  dC_{\ast}(\phi)(\sigma) = &d\phi(\sigma) \\
    = &\sum\limits_{i=0}^{n}\xi^{i}\langle \phi(v_{0}),\dots,\hat{\phi(v_{i})},\dots,\phi(v_{n})\rangle\\
    = &\phi(\sum\limits_{i=0}^{n}\xi^{i}\langle v_{0},\dots,\hat{v_{i}},\dots,v_{n}\rangle)\\
    = & C_{\ast}(\phi)(d\sigma).
\end{split}
\end{equation}
Obviously, $C_{\ast}(\phi)$ preserves identity. The desired result follows.
\end{proof}

\begin{corollary}\label{corollary:functor}
The Mayer homology $H_{\ast,q}(-;\mathbb{C}):\mathbf{Cpx}\to \mathbf{Vec}_{\mathbb{K}}$ is a functor from the category of simplicial complexes to the category of vector spaces over $\mathbb{K}$.
\end{corollary}
\begin{proof}
It is a directed corollary of Proposition \ref{proposition:functor1} and Proposition \ref{proposition:functor2}.
\end{proof}

The generalized Mayer homology contains the information of the usual simplicial homology. It is worth noting that the Mayer homology here is different from the simplicial homology. Thus, we can obtain additional topological information from the Mayer homology defined above.
\begin{lemma}
Let $M_{n,q}$ be the representation matrix of $d_{n,q}=d_{n-q+1}\cdots d_{n-1}d_{n}:C_{n}(K;\mathbb{C})\to C_{n-q}(K;\mathbb{C})$. Then we have
\begin{equation}
  \beta_{n,q}=\dim C_{n}(K;\mathbb{C})- \rank (M_{n,q})-\rank (M_{n+N-q,N-q}).
\end{equation}
\end{lemma}
\begin{proof}
Consider the short exact sequence
\begin{equation}
\xymatrix{
  0\ar@{->}[r]& Z_{n,q}\ar@{^{(}->}[r]&C_{n}(K;\mathbb{C}) \ar@{->}[r]^-{{d_{n,q}}}& B_{n-q,N-q} \ar@{->}[r]&0.
  }
\end{equation}
Indeed, we have the decomposition
\begin{equation}
  C_{n}(K;\mathbb{C})\cong Z_{n,q}\oplus B_{n-q,N-q}\cong H_{n,q}(K;\mathbb{C})\oplus B_{n,q} \oplus B_{n-q,N-q}.
\end{equation}
Note that $\rank (M_{n,q})=\dim B_{n-q,N-q}$. It follows that $\dim B_{n,q}=\rank (M_{n+N-q,N-q})$. Thus we have
\begin{equation}
  \dim C_{n}(K;\mathbb{C})=\beta_{n,q} +\rank (M_{n+N-q,N-q})+\rank (M_{n,q}).
\end{equation}
The desired result follows.
\end{proof}

\begin{example}
Consider the simplicial complex $\Delta[3]$ with the simplices
\begin{equation}
\begin{split}
   &  \{0\},\{1\},\{2\},\{3\}, \\
    &  \{0,1\},\{0,2\},\{0,3\},\{1,2\},\{1,3\},\{2,3\},\\
    &\{0,1,2\},\{0,1,3\},\{0,2,3\},\{1,2,3\},\\
    &\{0,1,2,3\}.
\end{split}
\end{equation}
Consider the 3-chain complex $C_{\ast}(\Delta[3];\mathbb{C})$ with the 3-boundary operator given by
\begin{equation}
\begin{split}
d_{3}\{0,1,2,3\}&=\{1,2,3\}+\xi\{0,2,3\}+\xi^{2}\{0,1,3\}+\{0,1,2\},\\
   d_{2}\{0,1,2\}&=\{1,2\}+\xi \{0,2\} + \xi^{2}\{0,1\},\\
   d_{2}\{0,1,3\}&=\{1,3\}+\xi \{0,3\} + \xi^{2}\{0,1\},\\
   d_{2}\{0,2,3\}&=\{2,3\}+\xi \{0,3\} + \xi^{2}\{0,2\},\\
   d_{2}\{1,2,3\}&=\{2,3\}+\xi \{1,3\} + \xi^{2}\{1,2\}\\
\end{split}
\end{equation}
and $d_{1}\{v,w\}=\{w\}+\xi\{v\}$ for $0\leq v<w\leq 3$. The representation matrices of $d_{1}$, $d_{2}$ and $d_{3}$ with the simplices as basis are given by
\begin{equation}
  B_{1} = \left(
    \begin{array}{cccc}
      \xi & 1 & 0 & 0 \\
      \xi & 0 & 1 & 0 \\
      \xi & 0 & 0 & 1 \\
      0 & \xi & 1 & 0 \\
      0 & \xi & 0 & 1 \\
      0 & 0 & \xi & 1 \\
    \end{array}
  \right),\quad B_{2}=\left(
                        \begin{array}{cccccc}
                          \xi^{2} & \xi & 0 & 1 & 0 & 0 \\
                          \xi^{2} & 0 & \xi & 0 & 1 & 0 \\
                          0 & \xi^{2} & \xi & 0 & 0 & 1 \\
                          0 & 0 & 0 & \xi^{2} & \xi & 1 \\
                        \end{array}
                      \right),\quad B_{3}=\left(
                                            \begin{array}{cccc}
                                              1 & \xi^{2} & \xi & 1 \\
                                            \end{array}
                                          \right).
\end{equation}
The representation matrices of $d_{1}d_{2}$ and $d_{2}d_{3}$ are listed as follows.
\begin{equation*}
  B_{2}B_{1}=\left(
               \begin{array}{cccc}
                 -\xi & -1 & -\xi^{2} & 0 \\
                 -\xi & -1 & 0 & -\xi^{2} \\
                 -\xi & 0 & -1 & -\xi^{2} \\
                 0 & -\xi & -1 & -\xi^{2} \\
               \end{array}
             \right),\quad B_{3}B_{2}=\left(
                                        \begin{array}{cccccc}
                                          -1 & -\xi^{2} & -\xi & -\xi & -1 &-\xi^{2} \\
                                        \end{array}
                                      \right).
\end{equation*}
Moreover, have have that
\begin{equation*}
  B_{3}B_{2}B_{1}=\mathbf{O}_{4\times 4},
\end{equation*}
which shows that $d^{3}=0$ on $C_{\ast}(\Delta[3];\mathbb{C})$. On the other hand, a straightforward calculation shows that
\begin{equation}
  \begin{split}
    Z_{3,1} & =Z_{3,2}= Z_{2,1}= B_{2,1}=0 , \\
    Z_{2,2} & =\mathrm{span}\{\{1,2,3\}+\xi\{0,2,3\}+\xi^{2}\{0,1,3\}+\{0,1,2\}\},\\
    B_{2,2} & =\mathrm{span}\{\{1,2,3\}+\xi\{0,2,3\}+\xi^{2}\{0,1,3\}+\{0,1,2\}\}, \\
    Z_{1,1} & =\mathrm{span}\{\{0,2\}-\{0,3\}-\{1,2\}+\{1,3\},\xi \{0,1\}-\xi \{0,2\}-\{1,3\}+\{2,3\}\}, \\
    B_{1,1} & =\mathrm{span}\{\xi\{0,1\}+\{0,2\}+\xi^{2}\{0,3\}+\xi^{2}\{1,2\}+\xi \{1,3\}+\{2,3\}\}, \\
    Z_{1,2} & =\mathrm{span}\{\{0,1\},\{0,2\},\{0,3\},\{1,2\},\{1,3\},\{2,3\}\}, \\
    B_{1,2} & =\mathrm{span}\{\{1,2\}+\xi \{0,2\} + \xi^{2}\{0,1\},\{1,3\}+\xi \{0,3\}+ \xi^{2}\{0,1\}, \\
    &\qquad\qquad\{2,3\}+\xi \{0,3\} + \xi^{2}\{0,2\},\{2,3\}+\xi \{1,3\} + \xi^{2}\{1,2\}\}, \\
    Z_{0,1} &=\mathrm{span}\{\{0\},\{1\},\{2\},\{3\}\}, \\
    B_{0,1} &=\mathrm{span}\{\{0\}-\{1\},\{1\}-\{2\},\{2\}-\{3\}\},\\
    Z_{0,2} &=\mathrm{span}\{\{0\},\{1\},\{2\},\{3\}\},\\
    B_{0,2} &=\mathrm{span}\{\{0\},\{1\},\{2\},\{3\}\},\\
  \end{split}
\end{equation}
By definition, one has
\begin{equation}
  H_{3,1}(\Delta[3];\mathbb{C})=H_{3,2}(\Delta[3];\mathbb{C})=H_{2,2}(\Delta[3];\mathbb{C})=H_{2,1}(\Delta[3];\mathbb{C})=H_{0,2}(\Delta[3];\mathbb{C})=0
\end{equation}
and
\begin{equation}
  H_{1,1}(\Delta[3];\mathbb{C})\cong \mathbb{C},\quad H_{1,2}(\Delta[3];\mathbb{C})\cong\mathbb{C}^{2},\quad H_{0,1}(\Delta[3];\mathbb{C})\cong \mathbb{C}.
\end{equation}
However, the simplicial homology of $\Delta[3]$ is $H_{n}(\Delta[3];\mathbb{C})=\left\{
                                                                                  \begin{array}{ll}
                                                                                    \mathbb{C}, & \hbox{$n=0$;} \\
                                                                                    0, & \hbox{\rm otherwise.}
                                                                                  \end{array}
                                                                                \right.
$
This indicates that even for contractible spaces, Mayer-Vietoris homology may not be trivial.
\end{example}

\begin{example}\label{example:mayer_betti}
Many common geometric shapes can be viewed as simplicial complexes through simplicial triangulations. In this example, we compute the Mayer Betti numbers for the simplicial complexes $\Delta[3]$, $\partial\Delta[3]$, and a hexagon. Additionally, we perform simplicial triangulations for the M\"{o}bius strip, torus, and octahedron, and calculate the Mayer Betti numbers for these simplicial complexes. The simplicial complex $\partial \Delta[3]$ has the simplices listed as follows:
\begin{equation}
\begin{split}
   &  \{0\},\{1\},\{2\},\{3\}, \\
    &  \{0,1\},\{0,2\},\{0,3\},\{1,2\},\{1,3\},\{2,3\},\\
    &\{0,1,2\},\{0,1,3\},\{0,2,3\},\{1,2,3\}.
\end{split}
\end{equation}
A hexagon is a simplicial complex with the simplices listed as follows:
\begin{equation}
  \begin{split}
   &  \{0\},\{1\},\{2\},\{3\},\{4\},\{5\}, \\
    &  \{0,1\},\{1,2\},\{2,3\},\{3,4\},\{4,5\},\{0,5\}.
\end{split}
\end{equation}
Now, we provide simplicial triangulations for the M\"{o}bius strip, torus, and octahedron, and compute the corresponding Mayer Betti numbers.

\begin{figure}[H]
  \centering
  \includegraphics[width=0.4\textwidth]{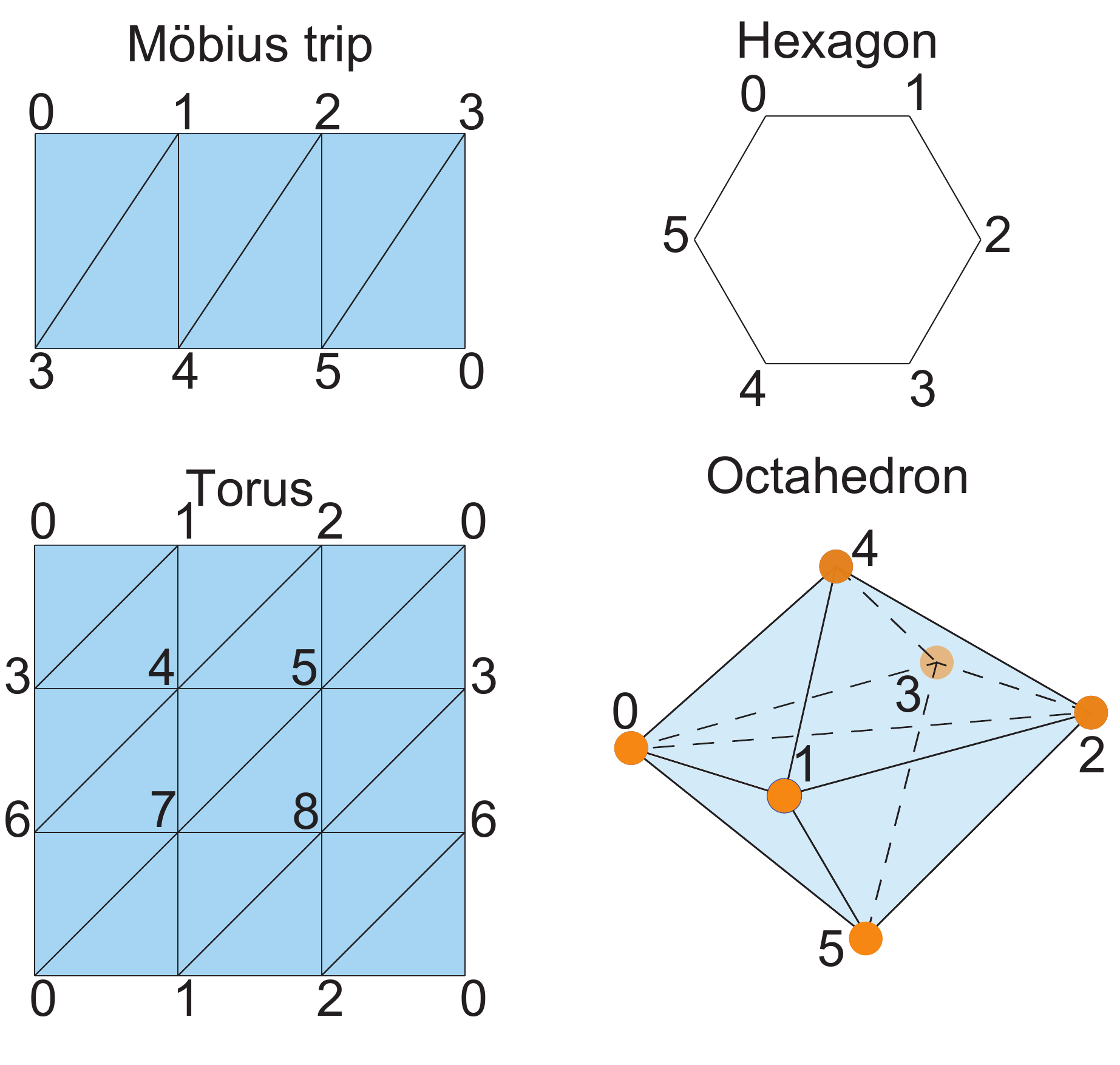}\qquad\qquad
  \caption{The simplicial triangulations of the M\"{o}bius strip, hexagon, torus, and octahedron.}\label{figure:complexes}
\end{figure}
\noindent The simplicial triangulations of the M\"{o}bius strip:
\begin{equation}
  \begin{split}
   &  \{0\},\{1\},\{2\},\{3\},\{4\},\{5\}, \\
    &  \{0,1\},\{0,5\},\{1,2\},\{2,3\},\{0,3\},\{1,3\},\{1,4\},\{2,4\},\{2,5\},\{3,4\},\{3,5\},\{4,5\},\\
    & \{0,1,3\},\{1,3,4\},\{1,2,4\},\{2,4,5\},\{2,3,5\},\{0,3,5\}.
\end{split}
\end{equation}
\noindent The simplicial triangulations of the torus:
\begin{equation}
  \begin{split}
   &  \{0\},\{1\},\{2\},\{3\},\{4\},\{5\},\{6\},\{7\},\{8\}, \\
    & \{0,1\},\{0,2\},\{0,3\},\{0,5\},\{0,6\},\{0,7\},\{1,2\},\{1,3\},\{1,4\},\{1,7\},\{1,8\},\{2,4\},\{2,5\},\{2,6\},\{2,8\},\\
    &\{3,4\},\{3,5\},\{3,6\},\{3,8\},\{4,5\},\{4,6\},\{4,7\},\{5,7\},\{5,8\},\{6,7\},\{6,8\},\{7,8\},\\
    & \{0,1,3\},\{0,1,7\},\{0,2,5\},\{0,2,6\},\{0,3,5\},\{0,6,7\},\{1,2,4\},\{1,3,4\},\{1,2,8\},\{1,7,8\},\\
    &  \{2,4,5\},\{2,6,8\},\{3,5,8\}, \{3,4,6\},\{3,6,8\},\{4,5,7\},\{4,6,7\},\{5,7,8\}.
\end{split}
\end{equation}
\noindent The simplicial triangulations of the octahedron:
\begin{equation}
  \begin{split}
   &  \{0\},\{1\},\{2\},\{3\},\{4\},\{5\}, \\
    &  \{0,1\},\{1,2\},\{2,3\},\{0,3\},\{0,4\},\{1,4\},\{2,4\},\{3,4\},\{0,5\},\{1,5\},\{2,5\},\{3,5\},\{0,2\}\\
    & \{0,1,4\},\{0,1,5\},\{1,2,4\},\{1,2,5\},\{2,3,4\}, \{2,3,5\},\{0,3,4\},\{0,3,5\},\{0,1,2\},\{0,2,3\},\{0,2,4\},\{0,2,5\},\\
    & \{0,1,2,4\},\{0,1,2,5\},\{0,2,3,4\},\{0,2,3,5\}.
\end{split}
\end{equation}
Using our algorithm's computations, Mayer Betti numbers can be obtained, as illustrated in Table \ref{table:complexes_betti}.
\begin{table}[H]
  \centering
  \caption{The Mayer Betti numbers for the simplicial complexes $\Delta[3]$, $\partial\Delta[3]$, a hexagon, and the Mayer Betti numbers resulting from the simplicial triangulations of the M\"{o}bius strip, torus, and octahedron.}\label{table:complexes_betti}
  \begin{tabular}{c|c|c|c|c|c|c}
  \hline
 simplicial complexes & $\beta_{0,1}$ & $\beta_{1,1}$ & $\beta_{2,1}$ & $\beta_{0,2}$ & $\beta_{1,2}$ & $\beta_{2,2}$\\
  \hline
  $\Delta[3]$ & 1 & 1 & 0 & 0 & 2 & 0\\
  $\partial\Delta[3]$ & 1 & 2 & 0 & 0 & 2 & 1\\
  Hexagon & 6 & 0 & 0 & 0 & 6 & 0\\
  M\"{o}bius trip& 1 & 6 & 0 & 0 & 6 & 1\\
  Torus & 1 & 18 & 0 & 0 & 9 & 10\\
  Octahedron & 1 & 3 & 1 & 0 & 2 & 3\\
  \hline
\end{tabular}
\end{table}
\end{example}

\subsection{The Mayer Laplacians on $N$-chain complexes}
Now, let $K$ be a simplicial complex. Then we have a chain complex $(C_{\ast}(K;\mathbb{C}),d)$. One can endow $C_{\ast}(K;\mathbb{C})$ with an inner product given by
\begin{equation}
  \langle \lambda\sigma,\mu\tau\rangle=\left\{
                               \begin{array}{ll}
                                 \lambda\cdot\overline{\mu}, & \hbox{$\sigma=\tau$;} \\
                                 0, & \hbox{\rm otherwise.}
                               \end{array}
                             \right.
\end{equation}
Here, $\lambda,\mu\in \mathbb{C}$, and $\overline{\mu}$ is the complex conjugate of $\mu$.
Consider the adjoint functor $d^{\ast}$ of $d$, i.e.,
\begin{equation}
  \langle dx,y\rangle=\langle x,d^{\ast}y\rangle
\end{equation}
for any $x,y\in C_{\ast}(K;\mathbb{C})$. Note that
\begin{equation}
  \langle d^{q}x,y\rangle=\langle d^{q-1}x,d^{\ast}y\rangle=\cdots=\langle x,(d^{\ast})^{q}y\rangle.
\end{equation}
By the definiteness of inner product, one has $(d^{q})^{\ast}=(d^{\ast})^{q}$.
For $1\leq q\leq N-1$, the \emph{Mayer Laplacian} $\Delta_{\ast,q}:C_{\ast}(K;\mathbb{C})\to C_{\ast}(K;\mathbb{C})$ is defined as
\begin{equation}
  \Delta_{\ast,q}:=(d^{q})^{\ast}\circ d^{q}+ d^{N-q}\circ (d^{N-q})^{\ast}.
\end{equation}
Choose the simplices of $K$ as an orthogonal basis of the $N$-chain complex $C_{\ast}(K;\mathbb{C})$ over $\mathbb{C}$. Let $B$ be the representation matrix of the linear operator $d:C_{\ast}(K;\mathbb{C})\to C_{\ast-1}(K;\mathbb{C})$ with respect to the chosen orthogonal basis under left multiplication. Then the representation matrix of $\Delta_{\ast,q}$ is given by
\begin{equation}
  L_{q}=B^{q}(\overline{B}^{q})^{T}+(\overline{B}^{N-q})^{T}B^{N-q}.
\end{equation}
Here, $\overline{B}^{T}$ is the conjugate transpose or Hermitian transpose matrix of $B$.
For the graded case, the Mayer Laplacian $\Delta_{n,q}:C_{n}(K;\mathbb{C})\to C_{n}(K;\mathbb{C})$ is given by
\begin{equation}
  \Delta_{n,q}=(d_{n})^{\ast}\circ\cdots\circ(d_{n-q+1})^{\ast} \circ d_{n-q+1}\circ\cdots\circ d_{n}+ d_{n+1}\circ\cdots\circ d_{n+N-q}\circ(d_{n+N-q})^{\ast}\circ\cdots\circ(d_{n+1})^{\ast}.
\end{equation}
Here, $d_{n}:C_{n}(K;\mathbb{C})\to C_{n-1}(K;\mathbb{C})$ is the operator of $d$ restricted to $C_{n}(K;\mathbb{C})$. Let $B_{n}$ be the representation matrix of $d_{n}$  with respect to the chosen orthogonal basis, and the representation matrix of $\Delta_{n,q}$ is given by
\begin{equation}
 L_{n,q}=B_{n}\cdots B_{n-q+1}\overline{B_{n-q+1}}^{T}\cdots \overline{B_{n}}^{T}+ \overline{B_{n+1}}^{T}\cdots \overline{B_{n+N-q}}^{T}B_{n+N-q} \cdots B_{n+1}.
\end{equation}
Here, $B_{n}$ is a complex matrix and $\overline{B_{n}}^{T}$ is the conjugate transpose of $B_{n}$.
\begin{proposition}\label{proposition:laplacian}
The Laplacian $\Delta_{n,q}$ on $C_{n}(K;\mathbb{C})$ is a self-adjoint and non-negative definite operator.
\end{proposition}
The proof of Proposition \ref{proposition:laplacian} is a straightforward verification, one can refer to \cite{chen2023persistent}. It is worth noting that even over the complex number field $\mathbb{C}$, the eigenvalues of the Laplacian operator are non-negative.

\begin{proposition}\label{proposition:equal_betti}
For any $n$ and $1\leq q\leq N-1$, we have $\dim\ker\Delta_{n,q}=\beta_{n,q}$.
\end{proposition}
\begin{proof}
It is a classic result. One can obtain a detailed proof in a \cite{liu2023algebraic}.
\end{proof}

\begin{example}
Let us compute the Mayer Laplacians on $\partial \Delta[3]$. We can obtain the $N$-chain complex $C_{\ast}(\partial \Delta[3];\mathbb{C})$ with the differential given by $d_{0}=0$,
\begin{equation}
  d_{1}\left(
                   \begin{array}{c}
                     \{0,1\} \\
                     \{0,2\} \\
                     \{0,3\} \\
                     \{1,2\} \\
                     \{1,3\} \\
                     \{2,3\} \\
                   \end{array}
                 \right)=\left(
                           \begin{array}{cccc}
                             \xi & 1 & 0 & 0 \\
                             \xi & 0 & 1 & 0 \\
                             \xi & 0 & 0 & 1 \\
                             0 & \xi & 1 & 0 \\
                             0 & \xi & 0 & 1 \\
                             0 & 0 & \xi & 1 \\
                           \end{array}
                         \right)\left(
                                  \begin{array}{c}
                                    \{0\} \\
                                    \{1\} \\
                                    \{2\} \\
                                    \{3\} \\
                                  \end{array}
                                \right)
\end{equation}
and
\begin{equation}
  d_{2}\left(
    \begin{array}{c}
      \{0,1,2\} \\
      \{0,1,3\} \\
      \{0,2,3\} \\
      \{1,2,3\} \\
    \end{array}
  \right)=\left(
            \begin{array}{cccccc}
              \xi^{2} & \xi & 0 & 1 & 0 & 0 \\
              \xi^{2} & 0 & \xi & 0 & 1 & 0 \\
              0 & \xi^{2} & \xi & 0 & 0 & 1 \\
              0 & 0 & 0 & \xi^{2} & \xi & 1 \\
            \end{array}
          \right)\left(
                   \begin{array}{c}
                     \{0,1\} \\
                     \{0,2\} \\
                     \{0,3\} \\
                     \{1,2\} \\
                     \{1,3\} \\
                     \{2,3\} \\
                   \end{array}
                 \right).
\end{equation}
We denote the representation matrix of $d_{n}$ by $B_{n}$. Observe that $B_{0}=B_{3}=0$. It follows that
\begin{equation}
  L_{0,1}=\overline{B_{1}}^{T}\overline{B_{2}}^{T}B_{2}B_{1}=\left(
                                                               \begin{array}{cccc}
                                                                 -\xi^{2} & -\xi^{2} & -\xi^{2} & 0 \\
                                                                 -1  & -1 & 0 & -\xi^{2} \\
                                                                 -\xi & 0 & -1 & -1 \\
                                                                 0 & -\xi & -\xi & -\xi \\
                                                               \end{array}
                                                             \right)\left(
                                                                      \begin{array}{cccc}
                                                                        -\xi & -1 & -\xi^{2} & 0 \\
                                                                        -\xi & -1 & 0 & -\xi^{2} \\
                                                                        -\xi & 0 & -1 & -\xi^{2} \\
                                                                        0 & -\xi & -1 & -\xi^{2} \\
                                                                      \end{array}
                                                                    \right)=\left(
                                                                              \begin{array}{cccc}
                                                                                3 & 2\xi^{2} & -1 & 2\xi \\
                                                                                2\xi & 3 & 2\xi^{2} & -1 \\
                                                                                -1 & 2\xi & 3 & 2\xi^{2} \\
                                                                                2\xi^{2} & -1 & 2\xi & 3 \\
                                                                              \end{array}
                                                                            \right),
\end{equation}
\begin{equation}
  L_{0,2}=\overline{B_{1}}^{T}B_{1}=\left(
                                      \begin{array}{cccccc}
                                        \xi^{2} & \xi^{2} & \xi^{2} & 0 & 0 & 0 \\
                                        1 & 0 & 0 & \xi^{2} & \xi^{2} & 0 \\
                                        0 & 1 & 0 & 1 & 0 & \xi^{2} \\
                                        0 & 0 & 1 & 0 & 1 & 1 \\
                                      \end{array}
                                    \right)\left(
                           \begin{array}{cccc}
                             \xi & 1 & 0 & 0 \\
                             \xi & 0 & 1 & 0 \\
                             \xi & 0 & 0 & 1 \\
                             0 & \xi & 1 & 0 \\
                             0 & \xi & 0 & 1 \\
                             0 & 0 & \xi & 1 \\
                           \end{array}
                         \right)=\left(
                                   \begin{array}{cccc}
                                     3 & \xi^{2} & \xi^{2} & \xi^{2} \\
                                     \xi & 3 & \xi^{2} & \xi^{2} \\
                                     \xi & \xi & 3 & \xi^{2} \\
                                     \xi & \xi & \xi & 3 \\
                                   \end{array}
                                 \right),
\end{equation}
\begin{equation}
    L_{1,1}=B_{1}\overline{B_{1}}^{T}+\overline{B_{2}}^{T}\overline{B_{3}}^{T}B_{3}B_{2}=\left(
                                                                                           \begin{array}{cccccc}
                                                                                             2 & 1 & 1 & \xi^{2} & \xi^{2} & 0 \\
                                                                                             1 & 2 & 1 & 1 & 0 & \xi^{2} \\
                                                                                             1 & 1 & 2 & 0 & 1 & 1 \\
                                                                                             \xi & 1 & 0 & 2 & 1 & \xi^{2} \\
                                                                                             \xi & 0 & 1 & 1 & 2 & 1 \\
                                                                                             0 & \xi & 1 & \xi & 1 & 2 \\
                                                                                           \end{array}
                                                                                         \right),
\end{equation}
\begin{equation}
    L_{1,2}=\overline{B_{2}}^{T}B_{2}=\left(
                                        \begin{array}{cccc}
                                          \xi & \xi & 0 & 0 \\
                                          \xi^{2} & 0 & \xi & 0 \\
                                          0 & \xi^{2} & \xi^{2} & 0 \\
                                          1 & 0 & 0 & \xi \\
                                          0 & 1 & 0 & \xi^{2} \\
                                          0 & 0 & 1 & 1 \\
                                        \end{array}
                                      \right)
    \left(
            \begin{array}{cccccc}
              \xi^{2} & \xi & 0 & 1 & 0 & 0 \\
              \xi^{2} & 0 & \xi & 0 & 1 & 0 \\
              0 & \xi^{2} & \xi & 0 & 0 & 1 \\
              0 & 0 & 0 & \xi^{2} & \xi & 1 \\
            \end{array}
          \right)=\left(
                    \begin{array}{cccccc}
                      2 & \xi^{2} & \xi^{2} & \xi & \xi & 0 \\
                      \xi & 2 & \xi^{2} & \xi^{2} & 0 & \xi \\
                      \xi & \xi & 2 & 0 & \xi^{2} & \xi^{2} \\
                      \xi^{2} & \xi & 0 & 2 & \xi^{2} & \xi \\
                      \xi^{2} & 0 & \xi & \xi & 2 & \xi^{2} \\
                      0 & \xi^{2} & \xi & \xi^{2} & \xi & 2 \\
                    \end{array}
                  \right).
\end{equation}
The spectra of $L_{0,1}$, $L_{0,2}$, $L_{1,1}$, and $L_{1,2}$ are
\begin{equation}
  \begin{split}
    \mathbf{Spec}(L_{0,1}) = & \{0,4-2\sqrt{3},4,4+2\sqrt{3}\}, \\
    \mathbf{Spec}(L_{0,2}) = & \{2-\sqrt{3},3,5,2+\sqrt{3}\},\\
    \mathbf{Spec}(L_{1,1}) = & \{0,0,2-\sqrt{3},3,5,2+\sqrt{3}\}, \\
    \mathbf{Spec}(L_{1,2}) = & \{0,0,2-\sqrt{3},3,2+\sqrt{3},5\}.\\
  \end{split}
\end{equation}
Let $N(\Delta_{n,q})$ denote the number of zero eigenvalues of the operator $\Delta_{n,q}$. It is worth noting that $N(\Delta_{0,1})=1$, $N(\Delta_{0,2})=0$, $N(\Delta_{1,1})=2$, $N(\Delta_{2,2})=2$. This is consistent with the Betti numbers corresponding to Table \ref{table:complexes_betti}.
\end{example}

\begin{example}\label{example:hexagon}
Now, we will compute the Mayer Laplacians of the hexagon. As described in Example \ref{example:mayer_betti}, the $3$-chain of a hexagon is a graded vector space with the corresponding $3$-differential given by
\begin{equation}
d_{1}\left(
       \begin{array}{c}
         \{0,1\} \\
         \{1,2\} \\
         \{2,3\} \\
         \{3,4\} \\
         \{4,5\} \\
         \{0,5\} \\
       \end{array}
     \right)=\left(
               \begin{array}{cccccc}
                 \xi & 1 & 0 & 0 & 0 & 0 \\
                 0 & \xi & 1 & 0 & 0 & 0 \\
                 0 & 0 & \xi & 1 & 0 & 0 \\
                 0 & 0 & 0 & \xi & 1 & 0 \\
                 0 & 0 & 0 & 0 & \xi & 1 \\
                 \xi & 0 & 0 & 0 & 0 & 1\\
               \end{array}
             \right)
     \left(
               \begin{array}{c}
                 \{0\} \\
                 \{1\}  \\
                 \{2\}  \\
                 \{3\}  \\
                 \{4\}  \\
                 \{5\}  \\
               \end{array}
             \right)
\end{equation}
and $d_{n}=0$ for $n\neq 1$. The calculation for $N=3$ is shown in Table \ref{table:hexagon_3}.
\begin{table}[H]
  \centering
  \caption{Illustration of Mayer Laplacians for $N=3$.}\label{table:hexagon_3}
  \begin{small}
  \begin{tabular}{c|c|c|c|c}
    \hline
    $n$ & $n=0$,$q=1$ & $n=0$,$q=2$& $n=1$,$q=1$ &$n=1$,$q=2$ \\
    \hline
    $L_{n,q}$ & $\mathbf{O}_{6\times 6}$
     & $\left(
          \begin{array}{cccccc}
            2 & \xi^{2} & 0 & 0 & 0 & 1 \\
            \xi & 2 & \xi^{2} & 0 & 0 & 0 \\
            0 & \xi & 2 & \xi^{2} & 0 & 0 \\
            0 & 0 & \xi & 2 & \xi^{2} & 0 \\
            0 & 0 & 0 & \xi & 2 & 1 \\
            1 & 0 & 0 & 0 & 1 & 2 \\
          \end{array}
        \right)
     $
     & $\left(
          \begin{array}{cccccc}
            2 & \xi^{2} & 0 & 0 & 0 & 1 \\
            \xi & 2 & \xi^{2} & 0 & 0 & 0 \\
            0 & \xi & 2 & \xi^{2} & 0 & 0 \\
            0 & 0 & \xi & 2 & \xi^{2} & 0 \\
            0 & 0 & 0 & \xi & 2 & 1 \\
            1 & 0 & 0 & 0 & 1 & 2 \\
          \end{array}
        \right)
     $& $\mathbf{O}_{6\times 6}$\\
         \hline
    $\beta_{n,q}$ & 6 & 0 & 0 &6\\
      \hline
    $\mathbf{Spec}(L_{n,q})$ &\{0,0,0,0,0,0\}& \{0.12,0.47,1.65,2.35,3.53,3.88\} & \{0.12,0.47,1.65,2.35,3.53,3.88\} &\{0,0,0,0,0,0\}\\
    \hline
  \end{tabular}
  \end{small}
  \end{table}
\noindent For the case $N=5$, we have the corresponding $5$-differential given by
\begin{equation}
d_{1}\left(
       \begin{array}{c}
         \{0,1\} \\
         \{1,2\} \\
         \{2,3\} \\
         \{3,4\} \\
         \{4,5\} \\
         \{0,5\} \\
       \end{array}
     \right)=\left(
               \begin{array}{cccccc}
                 \xi_{5} & 1 & 0 & 0 & 0 & 0 \\
                 0 & \xi_{5} & 1 & 0 & 0 & 0 \\
                 0 & 0 & \xi_{5} & 1 & 0 & 0 \\
                 0 & 0 & 0 & \xi_{5} & 1 & 0 \\
                 0 & 0 & 0 & 0 & \xi_{5} & 1 \\
                 \xi_{5} & 0 & 0 & 0 & 0 & 1 \\
               \end{array}
             \right)
     \left(
               \begin{array}{c}
                 \{0\} \\
                 \{1\}  \\
                 \{2\}  \\
                 \{3\}  \\
                 \{4\}  \\
                 \{5\}  \\
               \end{array}
             \right)
\end{equation}
and $d_{n}=0$ for $n\neq 1$. The calculated result at this point is shown in Table \ref{table:hexagon_5}.
\begin{table}[H]
  \centering
  \caption{Illustration of Mayer Laplacians for $N=5$.}\label{table:hexagon_5}
  \begin{small}
  \begin{tabular}{c|c|c|c|c}
    \hline
    $n$ & $n=0$,$q=1$ & $n=0$,$q=2$& $n=0$,$q=3$ &$n=0$,$q=4$ \\
    \hline
    $L_{n,q}$ & $\mathbf{O}_{6\times 6}$
     & $\mathbf{O}_{6\times 6}$
     & $\mathbf{O}_{6\times 6}$ & $\left(
          \begin{array}{cccccc}
            2 & \xi^{4}_{5} & 0 & 0 & 0 & 1 \\
            \xi_{5} & 2 & \xi^{4}_{5} & 0 & 0 & 0 \\
            0 & \xi_{5} & 2 & \xi^{4}_{5} & 0 & 0 \\
            0 & 0 & \xi_{5} & 2 & \xi^{4}_{5} & 0 \\
            0 & 0 & 0 & \xi_{5} & 2 & 1 \\
            1 & 0 & 0 & 0 & 1 & 2 \\
          \end{array}
        \right)$\\
         \hline
    $\beta_{n,q}$ & 6 & 6 & 6 &0\\
      \hline
    $\mathbf{Spec}(L_{n,q})$ &\{0,0,0,0,0,0\}& \{0,0,0,0,0,0\} & \{0,0,0,0,0,0\} &\{0.04,0.66,1.38,2.62,3.34,3.96\}\\
    \hline
  \end{tabular}
  \end{small}

  \begin{small}
  \begin{tabular}{c|c|c|c|c}
    \hline
    $n$ & $n=1$,$q=1$ & $n=1$,$q=2$& $n=1$,$q=3$ &$n=1$,$q=4$ \\
    \hline
    $L_{n,q}$ & $\left(
          \begin{array}{cccccc}
            2 & \xi^{4}_{5} & 0 & 0 & 0 & 1 \\
            \xi_{5} & 2 & \xi^{4}_{5} & 0 & 0 & 0 \\
            0 & \xi_{5} & 2 & \xi^{4}_{5} & 0 & 0 \\
            0 & 0 & \xi_{5} & 2 & \xi^{4}_{5} & 0 \\
            0 & 0 & 0 & \xi_{5} & 2 & 1 \\
            1 & 0 & 0 & 0 & 1 & 2 \\
          \end{array}
        \right)$
     & $\mathbf{O}_{6\times 6}$
     & $\mathbf{O}_{6\times 6}$ & $\mathbf{O}_{6\times 6}$\\
         \hline
    $\beta_{n,q}$ & 0 & 6 & 6 &6\\
      \hline
    $\mathbf{Spec}(L_{n,q})$ &\{0.04,0.66,1.38,2.62,3.34,3.96\}& \{0,0,0,0,0,0\} & \{0,0,0,0,0,0\} &\{0,0,0,0,0,0\}\\
    \hline
  \end{tabular}
  \end{small}
  \end{table}
\noindent Our calculations demonstrate that the eigenvalues are consistently non-negative definite. Moreover, the number of zero eigenvalues of Laplacians coincides with the corresponding Mayer Betti numbers.
\end{example}

In an intuitive sense, the Mayer homology and Mayer Laplacian of a complex reflect connections between simplices at different dimensions.
The corresponding Betti numbers reveal the topological cycles representing interactions between simplices of different dimensions, whereas the eigenvalues of the Laplacian operator deconstruct the connectivity between simplices of various dimensions. These relationships are more intricate and subtle, extending beyond what traditional simplicial homology theory can  capture.

\section{Persistence on Mayer invariants}\label{section:persistence_Mayer}

In this section, we will explore the persistent versions of Mayer homology and Mayer Laplacians. Since Mayer homology and Mayer Laplacians provide information different from the usual simplicial homology and Laplacian, investigating Mayer invariants is highly meaningful for our study of the topological characteristics and geometric structure of data. From now on, the ground field is taken to be the complex number field $\mathbb{C}$. Besides, we always consider the case that $N$ is a prime number for the sake of simplicity.

\subsection{Persistent Mayer homology}

Let $K$ be a simplicial complex, and let $f:K\to \mathbb{R}$ be a real-valued function defined on $K$ such that $f(\sigma)\leq f(\tau)$ for every face $\sigma$ of $\tau$ in $K$. For each real number $a$, we can obtain a sub complex $K_{a}=\{\sigma\in K|f(\sigma)\leq a\}$ of $K$. Moreover, for real numbers $a\leq b$, one has $K_{a}\subseteq K_{b}$. Thus, we can obtain a filtration of simplicial complexes
\begin{equation}
  K_{a_{1}}\subseteq K_{a_{2}}\subseteq \cdots \subseteq K_{a_{m}}
\end{equation}
for real numbers $a_{1}<a_{1}<\cdots< a_{m}$. By Proposition \ref{proposition:functor2}, we have a sequence of $N$-chain complexes
\begin{equation}
  C_{\ast}(K_{a_{1}};\mathbb{C})\to C_{\ast}(K_{a_{2}};\mathbb{C})\to \cdots \to C_{\ast}(K_{a_{m}};\mathbb{C}).
\end{equation}
By Proposition \ref{proposition:functor1}, this induces a sequence of Mayer homology
\begin{equation}
   H_{\ast,q}(K_{a_{1}};\mathbb{C})\to H_{\ast,q}(K_{a_{2}};\mathbb{C})\to \cdots \to H_{\ast,q}(K_{a_{m}};\mathbb{C})
\end{equation}
for any $1\leq q\leq N-1$. For any real numbers $a\leq b$ and $1\leq q\leq N-1$, the \emph{$(a,b)$-persistent Mayer homology} is defined by
\begin{equation}
  H_{n,q}^{a,b}:=\im (H_{n,q}(K_{a};\mathbb{C})\to H_{n,q}(K_{b};\mathbb{C})),\quad n\geq 0.
\end{equation}
The rank of $H_{n,q}^{a,b}$ is the $(a,b)$-persistent Betti numbers. The persistent Betti numbers can also be visualized using a persistence diagram or barcode. It is worth noting that for each $1\leq q\leq N-1$, we can obtain a persistence diagram, which means that the persistent Mayer homology contains more information than the usual persistent homology. Moreover, the fundamental theorems of persistent homology are also applicable to persistent Mayer homology.

Let $\{K_{a_{i}}\}_{i\geq 1}$ be a filtration of simplicial complexes. For each $i\geq 1$, we have the map $x:H_{\ast,q}(K_{a_{i}};\mathbb{C})\to H_{\ast,q}(K_{a_{i+1}};\mathbb{C})$ induced by $i\to i+1$. Consider the persistent homology, denoted as $\mathbf{H}_{q}=\bigoplus\limits_{i=1}^{\infty}H_{\ast,q}(K_{a_{i}};\mathbb{C})$, which encapsulates homological information from all time steps. Then one has a map $x:\mathbf{H}_{q}\to \mathbf{H}_{q}$, where $x$ map a generator at $a_{i}$ to a generator at $a_{i+1}$. Let $\mathbb{C}[x]$ be a polynomial ring over the complex number field $\mathbb{C}$. The space $\mathbf{H}_{q}$ is a left $\mathbb{C}[x]$-module given by
\begin{equation}
  \mathbb{C}[x]\times \mathbf{H}_{q}\to \mathbf{H}_{q},\quad (f(x),\alpha)\mapsto f(x)(\alpha).
\end{equation}
Moreover, the module structure theorem for persistent Mayer homology is established as follows.
\begin{theorem}
For a filtration of finite simplicial complexes $\{K_{a_{i}}\}_{i\geq 1}$, the corresponding persistent Mayer homology $\mathbf{H}_{q}$ has a decomposition as $\mathbb{C}[x]$-module
\begin{equation}
  \mathbf{H}_{q}\cong \left(\bigoplus\limits_{t} \mathbb{C}[x] \cdot \alpha_{b_{t}} \right)\oplus \left(\bigoplus\limits_{s}  \mathbb{C}[x]/x^{c_{s}}\cdot \beta_{b_{s}}\right).
\end{equation}
\end{theorem}
The proof of the above theorem is essentially a replica of the standard persistent homology structure theorem. Similarly, the generators in the free part, denoted as $\alpha_{b_{t}}$, refer to those generators born at time $b_{t}$ and persist until infinity, while $\beta_{b_{s}}$ represents the generators born at time $b_{s}$ and dead at time $b_{s}+c_{s}$. Similarly, we can define the barcode for persistent Mayer homology and give the fundamental characterization theorem for barcodes.

\subsection{Wasserstein distance for Mayer persistence diagrams}

Recall that the $r$-th Wasserstein distance of persistence diagrams is defined by
\begin{equation}
  W_{r}(\mathcal{D},\mathcal{D}')=\inf\limits_{\gamma:\mathcal{D}\to \mathcal{D}'}\left(\sum\limits_{x\in \mathcal{D}}\|x-\gamma(x)\|_{s}^{r}\right)^{1/r},
\end{equation}
where $\mathcal{D},\mathcal{D}'$ are persistence diagrams, $\|\cdot\|_{s}$ denotes the $L_{s}$-distance on a persistence diagram, and the infimum is taken over all matchings between $\mathcal{D}$ and $\mathcal{D}'$.

In the context of a filtration of simplicial complexes, a family of persistence diagrams $\mathcal{D}_{1},\dots,\mathcal{D}_{N-1}$ can be obtained for the persistent Mayer homology concerning the $p$-boundary operator. This collection is referred to as the Mayer persistence diagram. To formalize the relationship between these diagrams, we introduce the $r$-th Wasserstein distance for Mayer persistence diagrams, defined by
\begin{equation}
  W_{r}(\{\mathcal{D}_{q}\}_{1\leq q\leq N-1},\{\mathcal{D}_{q}'\}_{1\leq q\leq N-1})=\left(\sum_{q=1}^{N-1}W_{r}(\mathcal{D},\mathcal{D}')^{r}\right)^{1/r}.
\end{equation}
The case where $r=\infty$ is notably well-known. In this scenario, the Wasserstein distance reduces to the bottleneck distance:
\begin{equation}
  d_{B}(\{\mathcal{D}_{q}\}_{1\leq q\leq N-1},\{\mathcal{D}_{q}'\}_{1\leq q\leq N-1})=\sup_{1\leq q\leq N-1}\inf\limits_{\gamma:\mathcal{D}_{q}\to \mathcal{D}_{q}'}\sup\limits_{x\in \mathcal{D}_{q}}|x-\gamma(x)|.
\end{equation}
The real number field $\mathbb{R}$ can be regarded as a poset category with the real numbers as objects and the binary relations $\leq $ as morphisms.
Recall that an $\mathbb{R}$-indexed diagram $\mathcal{F}$ in a category $\mathfrak{C}$ is a functor $\mathcal{F}:\mathbb{R}\to \mathfrak{C}$ from the poset category $\mathbb{R}$ to the category $\mathfrak{C}$. Let $\mathcal{F}^{\mathbb{R}}$ be the category of $\mathbb{R}$-indexed diagrams in $\mathfrak{C}$. Let $\Sigma:\mathcal{F}^{\mathbb{R}}\to \mathcal{F}^{\mathbb{R}}$ be a functor on the category of $\mathbb{R}$-indexed diagrams given by $(\Sigma^{\varepsilon}\mathcal{F})(a) = \mathcal{F}(a+x)$.
\begin{definition}
Let $\mathcal{F}$ and $\mathcal{G}$ be two $\mathbb{R}$-indexed diagrams in a category $\mathfrak{C}$. We say $\mathcal{F}$ and $\mathcal{G}$ are $\varepsilon$-interleaved if there are natural transformations $\Phi: \mathcal{F}\to \Sigma \mathcal{G}$ and $\Psi:\mathcal{G}\to \Sigma \mathcal{F}$ such that $(\Sigma^{\varepsilon}\Psi)\circ \Phi=\Sigma^{2\varepsilon}|_{\mathcal{F}}$ and $(\Sigma^{\varepsilon}\Phi)\circ \Psi=\Sigma^{2\varepsilon}|_{\mathcal{G}}$.
\end{definition}

\begin{definition}
Let $\mathcal{F}$ and $\mathcal{G}$ be two $\mathbb{R}$-indexed diagrams in a category $\mathfrak{C}$. The interleaving distance between $\mathcal{F}$ and $\mathcal{G}$ is defined by
\begin{equation}
  d_{I}(\mathcal{F},\mathcal{G})=\inf\{\varepsilon\geq 0|\text{$\mathcal{F}$ and $\mathcal{G}$ are $\varepsilon$-interleaved}\}.
\end{equation}
\end{definition}

Let $f,g$ be two  real-valued functions defined on a simplicial complex $K$. Then one has two filtrations of simplicial complexes. Let $\|f-g\|_{\infty}=\sup\limits_{\sigma\in K} |f(\sigma)-g(\sigma)|$. Let $\mathcal{D}_{q}(K,f)$ and $\mathcal{D}_{q}(K,g)$ be the persistence diagrams of $K$ filtered by $f$ and $g$, respectively. We have the following result.
\begin{theorem}
$d_{B}(\{\mathcal{D}_{q}(K,f)\}_{1\leq q\leq N-1},\{\mathcal{D}_{q}(K,g)\}_{1\leq q\leq N-1})\leq \|f-g\|_{\infty}$.
\end{theorem}
\begin{proof}
We construct the proof based on the concepts developed in \cite{bubenik2014categorification, bauer2020persistence, bauer2013induced}. We consider Mayer persistent homology as the entities in the category $\mathbf{Vec}^{\mathbb{R}}$ of diagrams in the vector spaces category indexed by $\mathbb{R}$. Similarly, we regard Mayer persistence diagrams as the entities in the category $\mathbf{Mch}^{\mathbb{R}}$ of diagrams in the matching category indexed by $\mathbb{R}$. By \cite[Theorem 1.7]{bauer2020persistence} and \cite[Proposition 4.3]{bauer2020persistence}, one has
\begin{equation}
  d_{B}(\mathcal{D}_{q}(K,f),\mathcal{D}_{q}(K,g))=d_{I}(\mathbf{H}_{q}(K,f),\mathbf{H}_{q}(K,g))
\end{equation}
Here, $d_{I}$ denotes the interleaving distance for diagrams indexed by $\mathbb{R}$. For $(K,f)$, we have a diagram $K^{f}:\mathbb{R}\to \mathbf{Simp}$ in the category of simplicial complexes given by $K_{a}^{f}=\{\sigma\in K|f(\sigma)\leq a\}$. Let $\varepsilon=\|f-g\|_{\infty}$. Then there are inclusions of simplicial complexes $K_{a}^{f}\hookrightarrow K_{a+\varepsilon}^{g}$ and $K_{a}^{g}\hookrightarrow K_{a+\varepsilon}^{f}$ for any real number $a$. Thus one has natural transformations $\Phi:K_{\bullet}^{f}\hookrightarrow K_{\bullet+\varepsilon}^{g}$ and $\Psi:K_{\bullet}^{g}\hookrightarrow K_{\bullet+\varepsilon}^{f}$ of $\mathbb{R}$-indexed diagrams. Here, $K_{\bullet}(a)=K_{a}$.
By construction, we have
\begin{equation}
  (\Sigma^{\varepsilon}\Psi)\circ \Phi= \Sigma^{2\varepsilon}|_{K_{\bullet}^{f}}.
\end{equation}
Here, $\Sigma^{\varepsilon}\Psi:K_{\bullet+\varepsilon}^{g}\hookrightarrow K_{\bullet+2\varepsilon}^{f}$ is given by $(\Sigma^{\varepsilon}\Psi)(K_{\bullet+\varepsilon}^{g})(a)=K_{a+2\varepsilon}^{f}$ and $\Sigma^{2\varepsilon}|_{K_{\bullet}^{f}}:K_{\bullet}^{f}\to K_{\bullet+2\varepsilon}^{f}$ is given by $\Sigma^{2\varepsilon}|_{K_{\bullet}^{f}}(K_{\bullet}^{f})(a)=K_{a+2\varepsilon}^{f}$. Similarly, one has $(\Sigma^{\varepsilon}\Phi)\circ \Psi= \Sigma^{2\varepsilon}|_{K_{\bullet}^{g}}$. It follows that $K^{f}$ and $K^{g}$ are $\varepsilon$-interleaved. By definition, we have
\begin{equation}
  d_{I}(K^{f},K^{g})\leq \varepsilon.
\end{equation}
By \cite[Proposition 3.6]{bubenik2014categorification} and Corollary \ref{corollary:functor}, we have
\begin{equation}
  d_{I}(\mathbf{H}_{q}(K,f),\mathbf{H}_{q}(K,g))\leq d_{I}(K^{f},K^{g})\leq\varepsilon.
\end{equation}
It follows that
\begin{equation}
  d_{B}(\mathcal{D}_{q}(K,f),\mathcal{D}_{q}(K,g))\leq d_{I}(K^{f},K^{g})\leq\varepsilon.
\end{equation}
By the definition of bottleneck distance, one has
\begin{equation}
  d_{B}(\{\mathcal{D}_{q}(K,f)\}_{1\leq q\leq N-1},\{\mathcal{D}_{q}(K,g)\}_{1\leq q\leq N-1})\leq \|f-g\|_{\infty}.
\end{equation}
The desired result follows.
\end{proof}

The aforementioned conclusion establishes the stability of persistent Mayer Betti numbers under the bottleneck distance. This guarantees that the persistence of Mayer Betti numbers is a steadfast and resilient topological feature, resistant to noise.

\subsection{Persistent Mayer Laplacians}

Let $\{K_{a_{i}}\}_{i\geq 1}$ be a filtration of simplicial complexes. Endow $C_{\ast}(K_{a_{m}};\mathbb{C})$ with an inner product structure over $\mathbb{C}$. Consequently, as subspaces, each $C_{\ast}(K_{a_{i}};\mathbb{C})$ inherits the inner product structure of $C_{\ast}(K_{a_{m}};\mathbb{C})$.

Consider the inclusion $j_{a,b}:K_{a}\to K_{b}$ of simplicial complexes. By Proposition \ref{proposition:functor2}, we have a morphism $C_{\ast}(j_{a,b}):C_{\ast}(K_{a};\mathbb{C})\to C_{\ast}(K_{b};\mathbb{C})$ of $N$-chain complexes. For the sake of simplicity, we denote $C_{n}^{a}=C_{n}(K_{a};\mathbb{C})$ with the corresponding Mayer differential $d^{a}_{n}$, and denote $j_{n}^{a,b}=C_{n}(j_{a,b})$. Moreover, we denote $d_{n,q}^{a}=d^{a}_{n-q+1}\cdots d^{a}_{n-1}d^{a}_{n}:C_{n}^{a}\to C_{n-q}^{a}$. Let
\begin{equation}
  C_{n,q}^{a,b}=\{x\in C_{n}^{b}|d_{n,q}^{b}x\in C_{n-q}^{a}\},\quad 1\leq q\leq N-1.
\end{equation}
It follows that $C_{n,q}^{a,b}$ is a subspace of $C_{n}^{b}$ with the subspace inner product. Besides, we have a linear map $d_{n,q}^{a,b}:C_{n,q}^{a,b}\to C_{n-q}^{a}$ given by $d_{n,q}^{a,b}(x)=d_{n,q}^{b}x$.
\begin{equation}
    \xymatrix@=1cm{
    C_{n+N-q}^{a}\ar@{->}[rr]^{ d^{a}_{n+N-q,N-q}}\ar@{^{(}->}[dd]_{j_{n+N-q}^{a,b}}&&\quad C_{n}^{a}\quad\ar@<0.75ex>[rr]^-{{ d^{a}_{n,q}} } \ar@{^{(}->}[dd]^{j_{n}^{a,b}}\ar@<0.75ex>[ld]^{{ (d_{n+N-q,N-q}^{a,b})^{\ast}}}&&\quad C_{n-q}^{a}\ar@<0.75ex>[ll]^-{  {(d^{a}_{n,q})^{\ast}}}\ar@{^{(}->}[dd]^{j_{n-q}^{a,b}}\\
                           &C_{n+N-q,N-q}^{a,b}\ar@<0.75ex>[ru]^-{{ d_{n+N-q,N-q}^{a,b}} }\ar@{^{(}->}[ld]&&&                       \\
    C_{n+N-q}^{b}\ar@{->}[rr]^{ d^{b}_{n+N-q,N-q}}&&\quad C_{n}^{b}\quad\ar@{->}[rr]^{d^{b}_{n,q}}&&\quad C_{n-q}^{b}
    }
\end{equation}
The \emph{$(a,b)$-persistent Mayer Laplacian} $\Delta_{n,q}^{a,b}:C_{n}^{a}\to C_{n}^{a}$ is defined by
\begin{equation}
  \Delta_{n,q}^{a,b}:=(d^{a}_{n,q})^{\ast}\circ d^{a}_{n,q}+d_{n+N-q,N-q}^{a,b}\circ (d_{n+N-q,N-q}^{a,b})^{\ast}.
\end{equation}
In particular, if $n<q$, the persistent Mayer Laplacian is reduced to $\Delta_{n,q}^{a,b}=d_{n+N-q,N-q}^{a,b}\circ (d_{n+N-q,N-q}^{a,b})^{\ast}$. We arrange the positive eigenvalues of $\Delta_{n,q}^{a,b}$ in ascending order as follows:
\begin{equation}
  \lambda_{n,q}^{a,b}(1),\lambda_{n,q}^{a,b}(2),\dots,\lambda_{n,q}^{a,b}(r),
\end{equation}
where $r$ is the number of positive eigenvalues. Specifically, $\lambda_{n,q}^{a,b}(1)$ denotes the smallest positive eigenvalue, serving as the spectral gap and bearing close relevance to the Cheeger constant in geometry.

Recall that for simplicial homology, the harmonic component of the persistent Laplacian and persistent homology are isomorphic. Similarly, the harmonic component of the persistent Mayer Laplacian and persistent Mayer homology are also isomorphic. This is presented follows.

\begin{theorem}
For any $a\leq b$, we have an isomorphism $\ker \Delta_{n,q}^{a,b}\cong H_{n,q}^{a,b}$, where $n\geq 0$ and $1\leq q\leq N-1$.
\end{theorem}
\begin{proof}
Note that $d^{a}_{n,q}\circ d_{n+N-q,N-q}^{a,b}=0$. The result follows from \cite[Proposition 3.1]{liu2023algebraic}.
\end{proof}

The above theorem indicates that, within the Mayer homology theory, the persistent Mayer Laplacian contains more information than persistent Mayer homology. The persistent Mayer Laplacian reflects the geometric characteristics of complexes. It can be easily proven that the eigenvalues of the persistent Mayer Laplacian are non-negative. We arrange the positive eigenvalues in ascending order, denoting them as $\lambda_{n,q}(1),\dots,\lambda_{n,q}(r)$. Here, $r$ is the number of positive eigenvalues. Typically, attention is often focused on the smallest positive eigenvalue, the largest positive eigenvalue, the average value of eigenvalues, and similar information. In this paper, our examples and applications will involve computing the smallest eigenvalue.

\subsection{Mayer invariants on Vietoris-Rips complexes}

Let $X$ be a finite set of points embedded in Euclidean space. It is always possible to construct a filtration of simplicial complexes. Common constructions include Vietoris-Rips complexes, alpha complexes, cubical complexes, and others. These complexes offer diverse topological descriptions for datasets. Now, we will focus on exploring the Mayer invariants on Vietoris-Rips complexes.

Given a real number $\epsilon$, the Vietoris-Rips complex on $X$ is given by the simplicial complex
\begin{equation}
  \mathcal{VR}_{\epsilon}=\{\sigma\subseteq X| \text{every pair of points in $\sigma$ has a distance not larger than $\epsilon$}\}.
\end{equation}
From the Vietoris-Rips complex, one can derive the $N$-chain complex $C_{\ast}(\mathcal{VR}_{\epsilon};\mathbb{C})$. Furthermore, for any real numbers $\epsilon\leq \epsilon'$, the inclusion $\mathcal{VR}_{\epsilon}\hookrightarrow \mathcal{VR}_{\epsilon'}$ induces the inclusion $C_{\ast}(\mathcal{VR}_{\epsilon};\mathbb{C})\hookrightarrow C_{\ast}(\mathcal{VR}_{\epsilon'};\mathbb{C})$ of $N$-chain complexes. It leads the persistent Mayer homology
\begin{equation}
  H_{n,q}^{\epsilon,\epsilon'}=\im (H_{n,q}(\mathcal{VR}_{\epsilon};\mathbb{C})\to H_{n,q}(\mathcal{VR}_{\epsilon'};\mathbb{C})),\quad n\geq 0.
\end{equation}
and the persistent Mayer Laplacian based on the Vietoris-Rips complexes, serving as the primary tool in our work.

\paragraph{Example 1.}
Consider the example where $X_1$ consists of the following seven points on a plane
\begin{equation}
  (0,0),(1,1),(1,-1),(2, 1),(2.5,1.5),(2.5, 0.5),(3,1).
\end{equation}
\begin{figure}[H]
	\centering
	\includegraphics[width=0.8\textwidth]{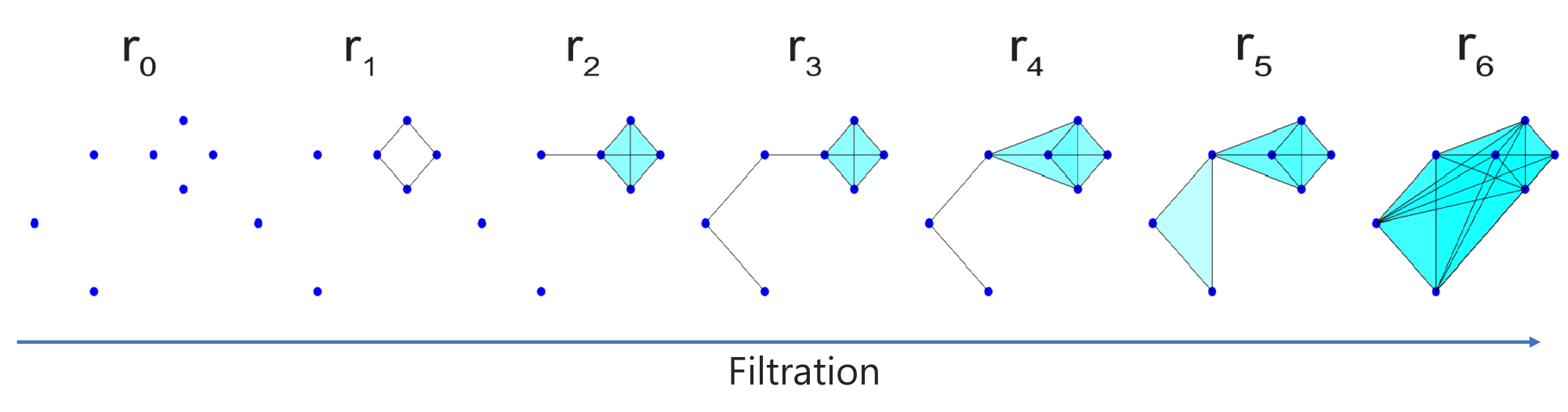}\\
	\caption{Illustration of the Vietoris-Rips complexes at different filtration radius for pointset $X_1$. Note that for the point set $X_1$ in this example, we can obtain a maximum of 12 Vietoris-Rips complexes with different filtration radius. For simplicity, we have omitted 5 complexes between $r_5$ and $r_6$. }
	\label{figure:persistent_trisqaure}
\end{figure}	

 Here, we exhibits a visualization of some of the corresponding Vietoris-Rips complexes in \autoref{figure:persistent_trisqaure}, labeled by their filtration radius, namely $r_0$ to $r_6$, respectively. In this example, the topological features we employed from the Mayer invariants include the Betti numbers at dimension $0$ and $1$. We display comparisons of calculation results of the persistent Mayer homology of the Vietoris-Rips complexes derived from the set $X$ with different $N$ values.

We first compare the case $N=2$ with $N=3$, shown in Figure \ref{figure:Mayer_Betti_2_3}. The $N=2$ case, which also represents the classical persistent Betti numbers, exhibit fewer topological features than the persistent Mayer Betti numbers for $N=3$ case. Specifically, the classical ($N=2$) persistent homology can yield non-trivial Betti numbers for dimensional $0$ and $1$ at filtration radius $r_0$,$r_1$,$r_2$, and $r_1$, respectively. In contrast, for $N=3$ case, the persistent Mayer homology reveals non-trivial Mayer Betti number 0 at $r_0$ ($q=1$ and $q=2$), $r_1$ ($q=1$ and $q=2$), $r_2$ ($q=1$ and $q=2$),  $r_3$ ($q=1$), $r_4$ ($q=1$), $r_5$ ($q=1$), and $r_6$ ($q=1$). Additionally, the $N=3$ case yields non-trivial Mayer Betti number 1 at $r_1$ ($q=1$ and $q=2$), $r_2$ ($q=1$ and $q=2$), $r_3$ ($q=1$ and $q=2$), $r_4$ ($q=1$ and $q=2$), $r_5$ ($q=1$ and $q=2$), and $r_6$ ($q=1$).

\begin{figure}[H]
	\centering
	\includegraphics[width=0.7\textwidth]{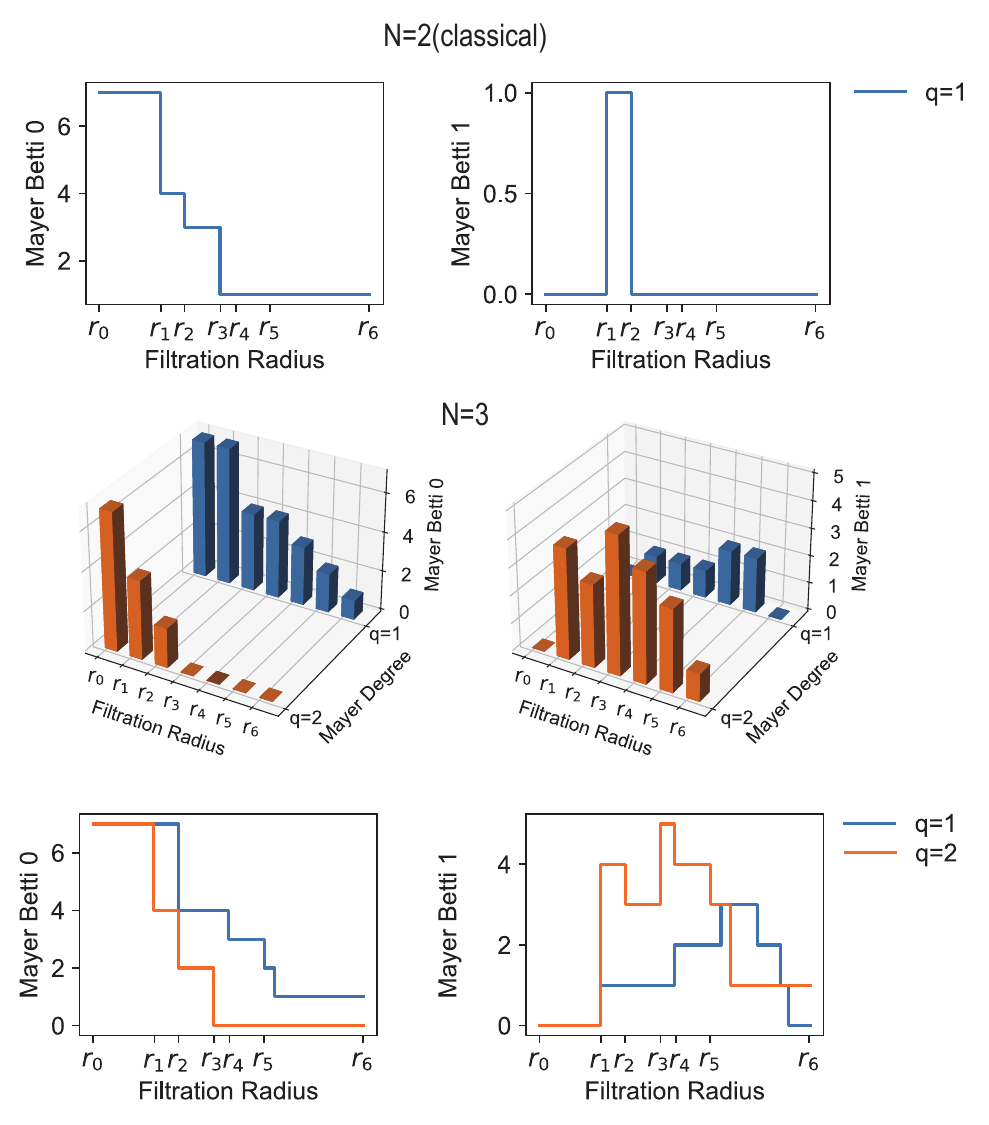}\\
	\caption{Comparison of persistent Betti numbers between the cases $N=2$, $N=3$.}\label{figure:Mayer_Betti_2_3}
\end{figure}

Moreover, persistent Mayer homology offers a more precise representation of the filtration variations across Mayer Betti numbers. Considering Mayer Betti number 0 for $N=3$, that of the Mayer degree $q=1$ demonstrates changes from $r_0$ to $r_1$, from $r_1$ to $r_2$, and from $r_2$ to $r_3$. For the Mayer degree $q=2$, it reveals changes from $r_1$ to $r_2$, from $r_3$ to $r_4$, from $r_4$ to $r_5$, and from $r_5$ to $r_6$. In the one-dimensional case of Mayer Betti numbers, the changes from $r_0$ to $r_1$, $r_3$ to $r_4$, $r_4$ to $r_5$, and $r_5$ to $r_6$ can be captured by either $q=1$ or $q=2$. Furthermore, additional variations, such as those from $r_1$ to $r_2$ and $r_2$ to $r_3$, are observable in the $q=2$ scenario. In contrast to traditional methods that can only capture changes from $r_0$ to $r_1$ and from $r_1$ to $r_2$, persistent Mayer homology demonstrates a more pronounced advantage by encompassing a broader range of changes, providing a richer and more detailed representation of the underlying geometric variations in the filtration.

\begin{figure}[H]
	\centering
	\includegraphics[width=0.7\textwidth]{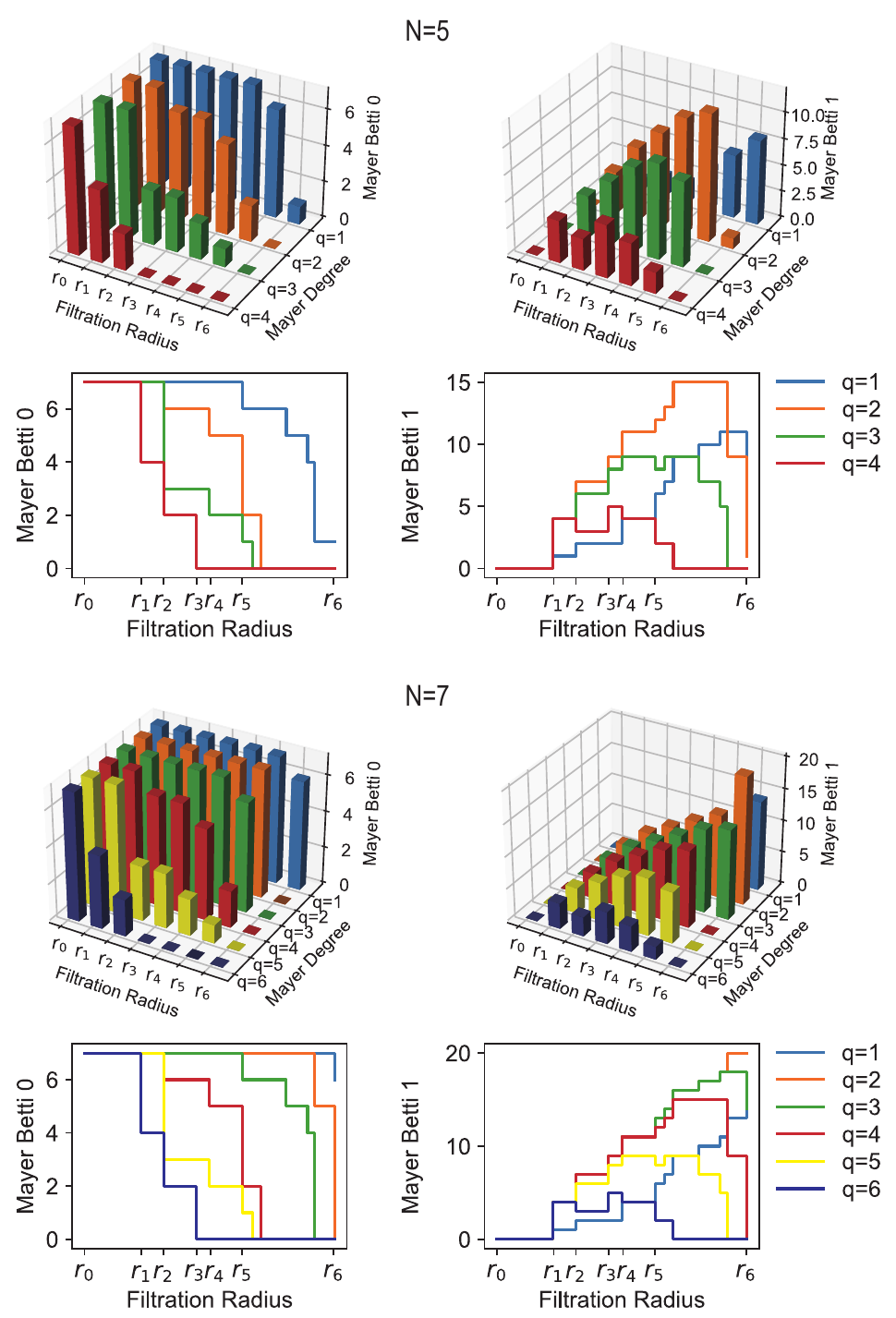}\\
	\caption{Illustration of persistent Betti numbers between the cases $N=5$, $N=7$. The Mayer degree, denoted by $q$, refers to the stage of Mayer homology.}\label{figure:Mayer_Betti_5_7}
\end{figure}

While in other cases, such as $N=5$, and $N=7$, more topological features are encompassed. As illustrated in \autoref{figure:Mayer_Betti_5_7}, we consistently observe $N-1$ Betti curves, each reflecting distinct topological information. To provide a more accurate description of the information content in the Betti curves obtained for different values of $N$, we conducted a statistical analysis of the variations in Betti 0 and Betti 1 for different values of $N$, shown in \autoref{table:betti_viaration}. We observe that with the increase in the value of $N$, the quantities of Betti 0 variations and Betti 1 variations strictly and positively increase. The increasing effect is more pronounced for Betti 1, indicating that, unlike the information obtained from the classical persistent homology of Rips complexes, the one-dimensional information provided by persistent Mayer homology also plays a crucial role.

Additionally, it is noteworthy that the average Betti variation in \autoref{table:betti_viaration} indicates that, for the majority of cases, increasing the value of $N$ not only results in obtaining more Betti curves but also enhances the topological information of each Betti curve. The only exception is the case of Betti 0 for $N=7$. This is primarily due to the fact that the point set considered in this example contains only 7 points, leading to a sparse existence of high-dimensional simplices in the corresponding Vietoris-Rips complex. In Mayer homology, Betti 0 variation implies that 0-dimensional simplices are killed by some higher-dimensional simplices. If the number of higher-dimensional simplices is too sparse, the difficulty of eliminating 0-dimensional simplices increases, leading to a reduction in the quantity of variations. However, in application scenarios, the number of points in the point set is generally much larger than the value of $N$. In such cases, we can typically expect an increase in the average Betti variations.

\begin{table}[H]
	\centering
	\begin{tabular}{c|c|c|c|c}
		\hline
		N value & Betti 0 variations &Avg. Betti 0 variations& Betti 1 variations&Avg. Betti 1 variations\\
		\hline
		2 & 3 &3 & 2& 2\\
		3 & 7 &3.5 & 12 & 6\\
		5 & 15 &3.75 & 33&8.25\\
		7 & 17 &2.83& 54& 9\\
		\hline
	\end{tabular}
	\caption{A statistics of the Mayer Betti curves variation for different $N$ value.}\label{table:betti_viaration}
\end{table}


\paragraph{Example 2.}

In the previous example, we have confirmed that Mayer Betti numbers can capture changes in the majority of Vietoris-Rips complexes. Therefore, a question worth discussing is whether Persistent Mayer Laplacian can provide more information compared to Persistent Mayer homology.

In this example, we show the comparison of Betti numbers and the smallest eigenvalues for the non-harmonic components of the Laplacians for the case $N=5$. Here, we consider example where points are distributed on the vertices of a three-dimensional cube. Let $X_2$ be a set with points given by
\begin{equation}
  (0, 0, 1.3),(0, 0, -1),(0, 1, 0),(0, -1, 0),(1, 0, 0),(-1, 0, 0).
\end{equation}
\autoref{figure:persistent_octahedron} shows the visualization of the Vietoris-Rips complexes.
\begin{figure}[H]
	\centering
	\includegraphics[width=0.6\textwidth]{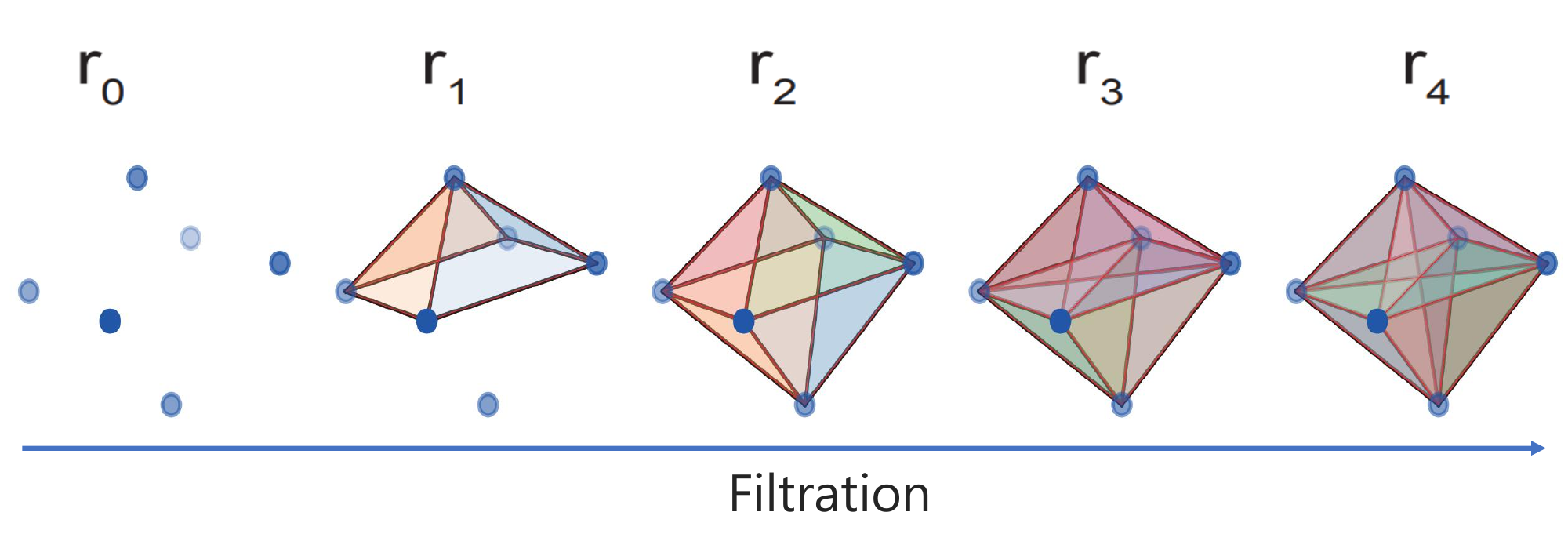}\\
	\caption{Illustration of the Vietoris-Rips complexes at different filtration radius for pointset $X_2$.}
	\label{figure:persistent_octahedron}
\end{figure}	

We are interested to know whether persistent Mayer Laplacian detects more geometric variations than  persistent Mayer Homology in characterizing data.
To this end, we compare the persistent Betti numbers and the smallest non-zero eigenvalues of persistent Mayer Laplacians derived from $X_2$ for the case $N=2$, $N=3$, and $N=5$ as shown in \autoref{figure:mayer_laplacian_2}, \autoref{figure:mayer_laplacian_3}, and \autoref{figure:mayer_laplacian_5}, respectively.  Since the harmonic spectra of persistent Mayer Laplacians fully recovery the topological information of persistent Mayer Homology, attention is given to whether
 Mayer Laplacian's non-zero eigenvalue can detect  additional variations
compared to   Mayer Betti numbers. Our results are summarized in  \autoref{table:laplacian_variation}. After comparison, we observe that the classical ($N=2$) Laplacian's nonharmonic spectra can detect more variations in both dimension 0 and 1. While Mayer Laplacian's first nonzero eigenvalue  is superior in dimension 0 for all $N=3$ cases, and $N=5, q=2$, $N=5, q=3$, $N=5, q=4$ cases, and in dimension 1 for $N=3, q=2$, $N=5, q=1$, and $N=5, q=4$ cases. It performs on par with Mayer Betti number in dimension 0 for $N=5, q=1$, in dimension 1 for $N=3, q=1$. In addition, Mayer Laplacian's first nonzero eigenvalue  captures fewer variations than  Mayer Betti number does in dimension 1 for $N=5, q=2$ and $N=5, q=3$. In summary, Mayer Laplacian exhibits superior performance compared to Mayer Betti numbers, confirming that persistent Mayer Laplacian indeed provides richer information compared to persistent Mayer Homology.

\begin{figure}[H]
	\centering
	\includegraphics[width=0.9\textwidth]{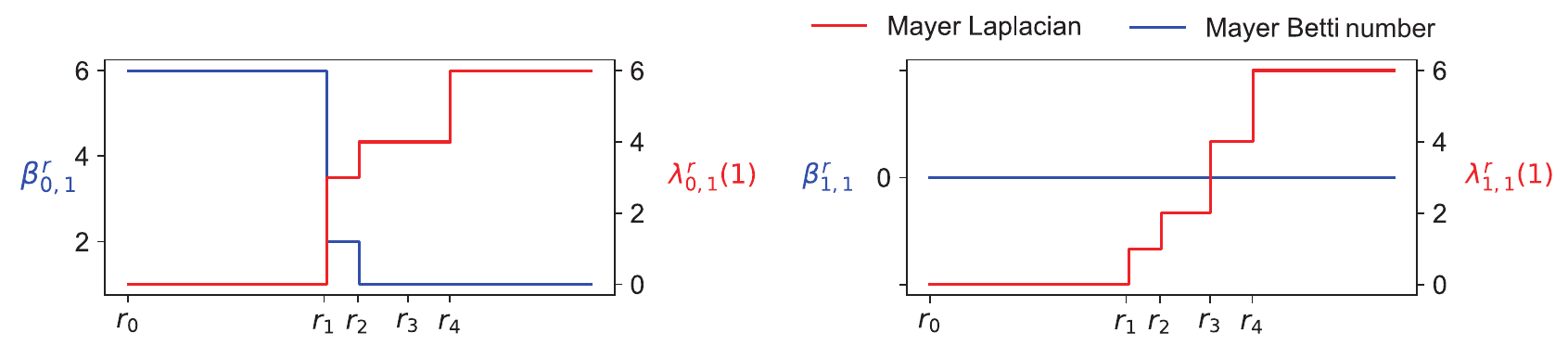}\\
	\caption{Comparison of persistent Betti numbers and the smallest positive eigenvalues of persistent Laplacians for the case that $N=2$ (classical). The blue curves denote the Betti curves, while the red curves represent changes of the smallest eigenvalues. The notion $\beta^{r}_{n,q}$ denotes the $n$-dimensional Betti number at stage $q$ of the Vietoris-Rips complex at distance $r$. The notion $\lambda_{n,q}^{r}(1)$ represents the smallest eigenvalue of the non-harmonic component of the Laplacian $\Delta_{n,q}^{r}$ at distance parameter $r$.}\label{figure:mayer_laplacian_2}
\end{figure}

\begin{figure}[H]
	\centering
	\includegraphics[width=0.9\textwidth]{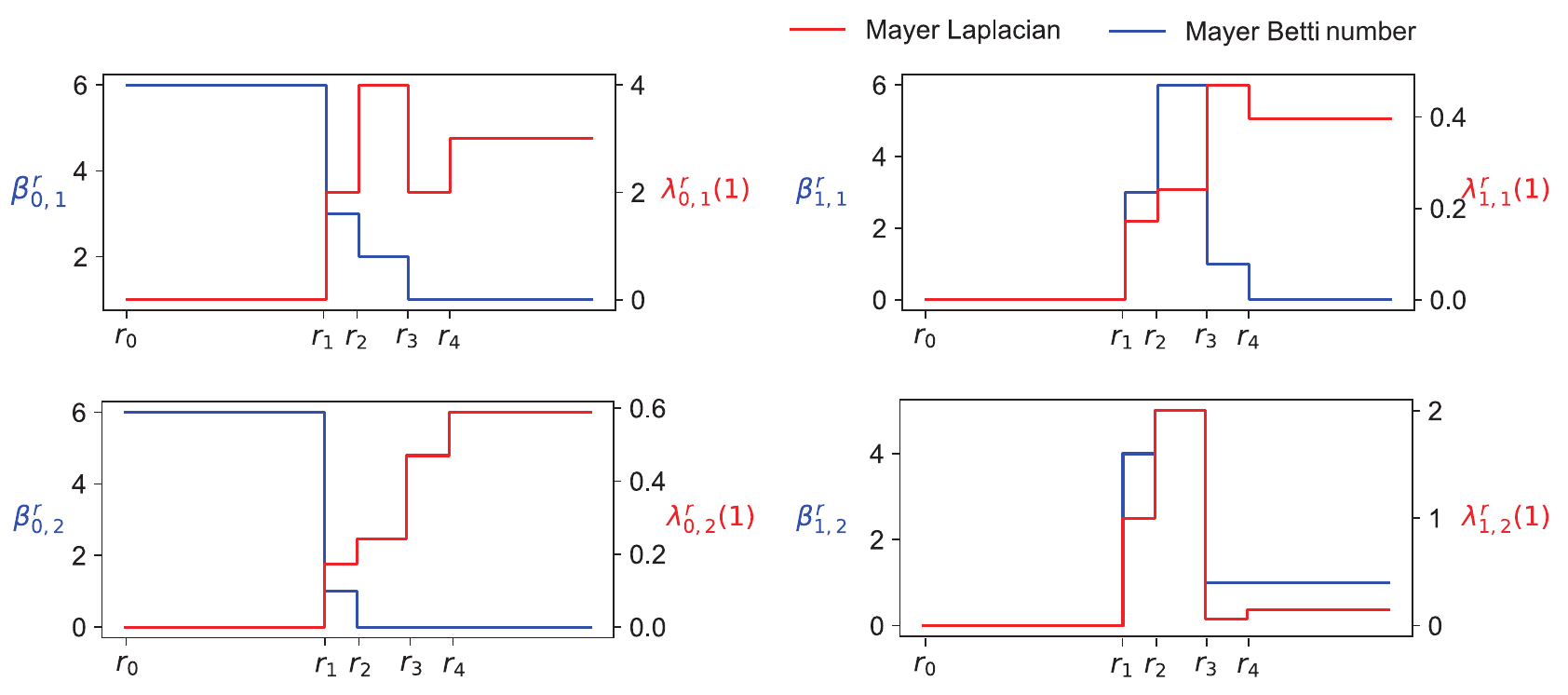}\\
	\caption{Comparison of persistent Betti numbers and the smallest positive eigenvalues of persistent Laplacians for the case that $N=3$. The blue curves denote the Betti curves, while the red curves represent changes of the smallest eigenvalues. The notion $\beta^{r}_{n,q}$ denotes the $n$-dimensional Betti number at stage $q$ of the Vietoris-Rips complex at distance $r$. The notion $\lambda_{n,q}^{r}(1)$ represents the smallest eigenvalue of the non-harmonic component of the Laplacian $\Delta_{n,q}^{r}$ at filtration  parameter $r$.}\label{figure:mayer_laplacian_3}
\end{figure}

\begin{figure}[H]
	\centering
	\includegraphics[width=0.9\textwidth]{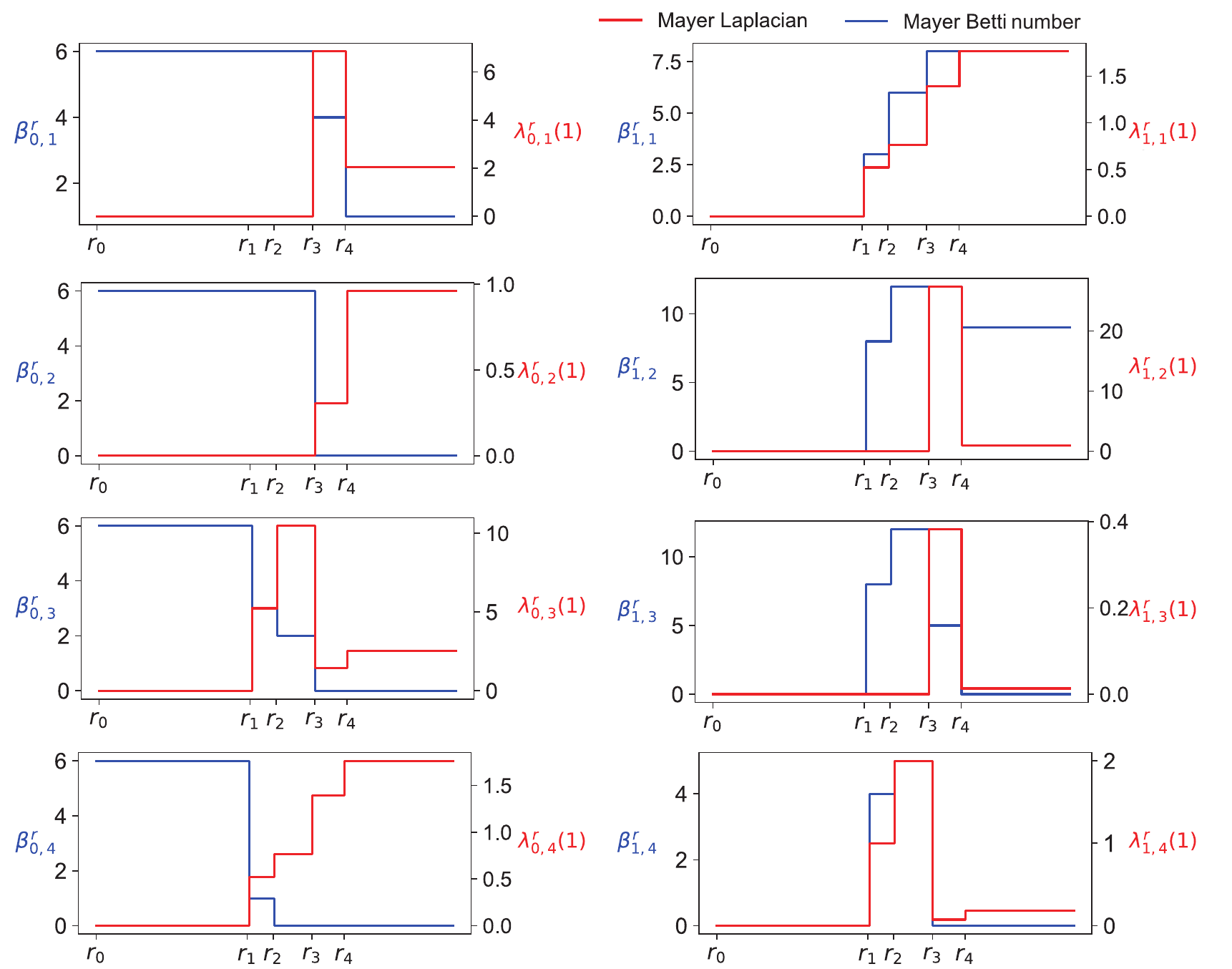}\\
	\caption{Comparison of persistent Betti numbers and the smallest positive eigenvalues of persistent Laplacians for the case that $N=5$. The blue curves denote the Betti curves, while the red curves represent changes of the smallest eigenvalues. The notion $\beta^{r}_{n,q}$ denotes the $n$-dimensional Betti number at stage $q$ of the Vietoris-Rips complex at distance $r$. The notion $\lambda_{n,q}^{r}(1)$ represents the smallest eigenvalue of the non-harmonic component of the Laplacian $\Delta_{n,q}^{r}$ at filtration  parameter $r$.}\label{figure:mayer_laplacian_5}
\end{figure}

A more detailed analysis reveals that the reason for the use of  Mayer Laplacian lies in its inability to detect the variations from $r_0$ to $r_1$ and from $r_1$ to $r_2$ in the 1-dimensional case for $N=5, q=2$ and $N=5, q=3$. In both of these scenarios, the smallest eigenvalues of persistent Laplacians are consistently 0. This indicates that, in these cases, all 1-dimensional simplices precisely serve as representatives of some Mayer homology classes. Therefore, we believe that while persistent Mayer Laplacian's first eigenvalue can offer more information compared to persistent Mayer homology, it is not sufficient to replace the latter. The combination of both harmonic and non-harmonic spectra  is necessary to achieve better results in practical applications.

\begin{table}[H]
	\centering
	\begin{tabular}{c|c c|c c}
		\hline
Mayer invariants & $\beta_{0,q}$ & $\lambda_{0,q}(1)$ & $\beta_{1,q}$ & $\lambda_{1,q}(1)$\\
		\hline \hline
$N=2$, $q=1$& 2            & 3             & 0                 & 4                 \\
\hline
$N=3$, $q=1$& 3            & 4             & 4                 & 4                 \\
$N=3$, $q=2$& 2            & 4             & 3                 & 4                 \\
\hline
$N=5$, $q=1$& 2            & 2             & 3                 & 4                 \\
$N=5$, $q=2$& 1            & 2             & 3                 & 2                 \\
$N=5$, $q=3$& 3            & 4             & 4                 & 2                 \\
$N=5$, $q=4$& 2            & 4             & 3                 & 4                 \\
\hline
	\end{tabular}
	\caption{A comparison of variation  detection of the Mayer Betti numbers with the Mayer Laplacian's first non-zero eigenvalues for $N=2, 3$, and $5$.}\label{table:laplacian_variation}
\end{table}


\section{Applications}\label{section:applications}

In this section, we will compute the persistent Mayer Betti numbers and spectral gaps of Mayer Laplacians for fullerene $\mathrm{C}_{60}$ and cucurbit[7]uril $\mathrm{CB}7$. We use the atomic coordinates of molecules as spatial points to construct the Vietoris-Rips complex, and then build an $N$-chain complex on it. Typically, we consider the cases $N=2$, $N=3$, and $N=5$. Here, $N$ represents the integer that $d^{N}=0$. We focus on the Mayer Betti numbers denoted as $\beta_{n,q}$ and the smallest positive eigenvalues of Mayer Laplacians (spectral gaps) denoted as $\lambda_{n,q}(1)$. In this work, $n$ denotes the dimension of Mayer homology or Mayer Laplacians, and we always compute the Mayer Betti numbers and the spectral gaps of Mayer Laplacians for dimensions 0 and 1. The parameter  $q$ refers to the subscript of Mayer homology or Mayer Laplacians, representing the $q$-th stage, where $1 \leq q \leq N-1$. Specifically, for the case of $N=2$, we obtain the usual simplicial homology and its corresponding Laplacian, where $q$ can only take the value of 1. This implies that for a given dimension $n$, there is only one homology group and one Laplacian operator.

In our examples, we compute the corresponding Mayer Betti numbers and spectral gaps of Mayer Laplacians for $N=2$, $N=3$, and $N=5$. Additionally, for a given $N$, we calculate the Mayer Betti numbers and spectral gaps of Mayer Laplacians for $1\leq q\leq N-1$. These computational results will reveal the distinctive properties of Mayer homology and Mayer Laplacian. Generally speaking, Mayer homology is a homotopy invariant, and in a certain sense, Mayer homology provides a more geometric characterization.

\begin{figure}[H]
	\centering
	\includegraphics[width=0.4\textwidth]{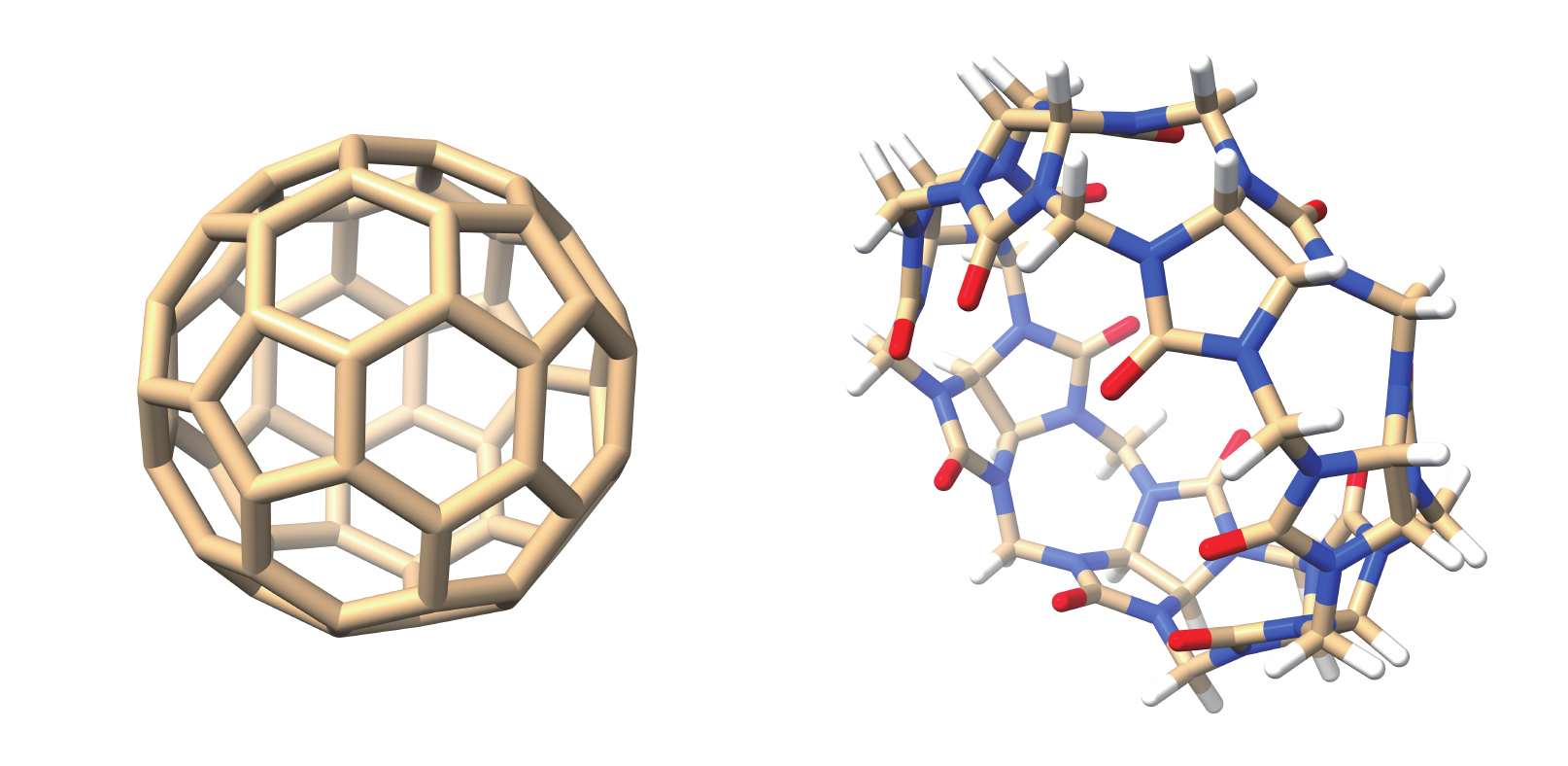}
	\caption{Structures of the fullerene $\mathrm{C}_{60}$ (Left) and  the cucurbit[7]uril $\mathrm{CB}7$ (Right).}\label{figure:molecules}
\end{figure}

In the depicted 3D structure showcased in \autoref{figure:molecules}, the fullerene $\mathrm{C}_{60}$ is presented as a carbon molecule with a distinctive soccer ball-like arrangement, comprising 60 carbon points. In contrast, the macrocyclic compound cucurbit[7]uril ($\mathrm{CB}7$) is intricately composed of 126 points, encompassing carbon, hydrogen, oxygen, and nitrogen atoms. Given the more symmetrical and concise configuration of $\mathrm{C}_{60}$ in comparison to the complex structure of $\mathrm{CB}7$, an effective featurization method is anticipated to reveal more nuanced patterns for $\mathrm{CB}7$.

In \autoref{figure:C60_b} and \autoref{figure:C60_l}, as well as \autoref{figure:CB7_b} and \autoref{figure:CB7_l}, distinct colors represent the numerical values of different Betti numbers and spectral gaps. The structural differences between $\mathrm{C}_{60}$ and $\mathrm{CB}7$ are readily apparent from the comparisons in \autoref{figure:C60_b} with \autoref{figure:CB7_b}, and \autoref{figure:C60_l} with \autoref{figure:CB7_l}. The persistent Mayer Betti numbers and persistent Mayer Laplacians of CB7 display more intricate patterns, and the critical points of variation in these patterns involve a broader range of filtration radius. This highlights the potential of persistent Mayer homology and persistent Mayer Laplacian as highly effective tools for featuring molecular structures.
\begin{figure}[H]
	\centering
	\includegraphics[width=0.8\textwidth]{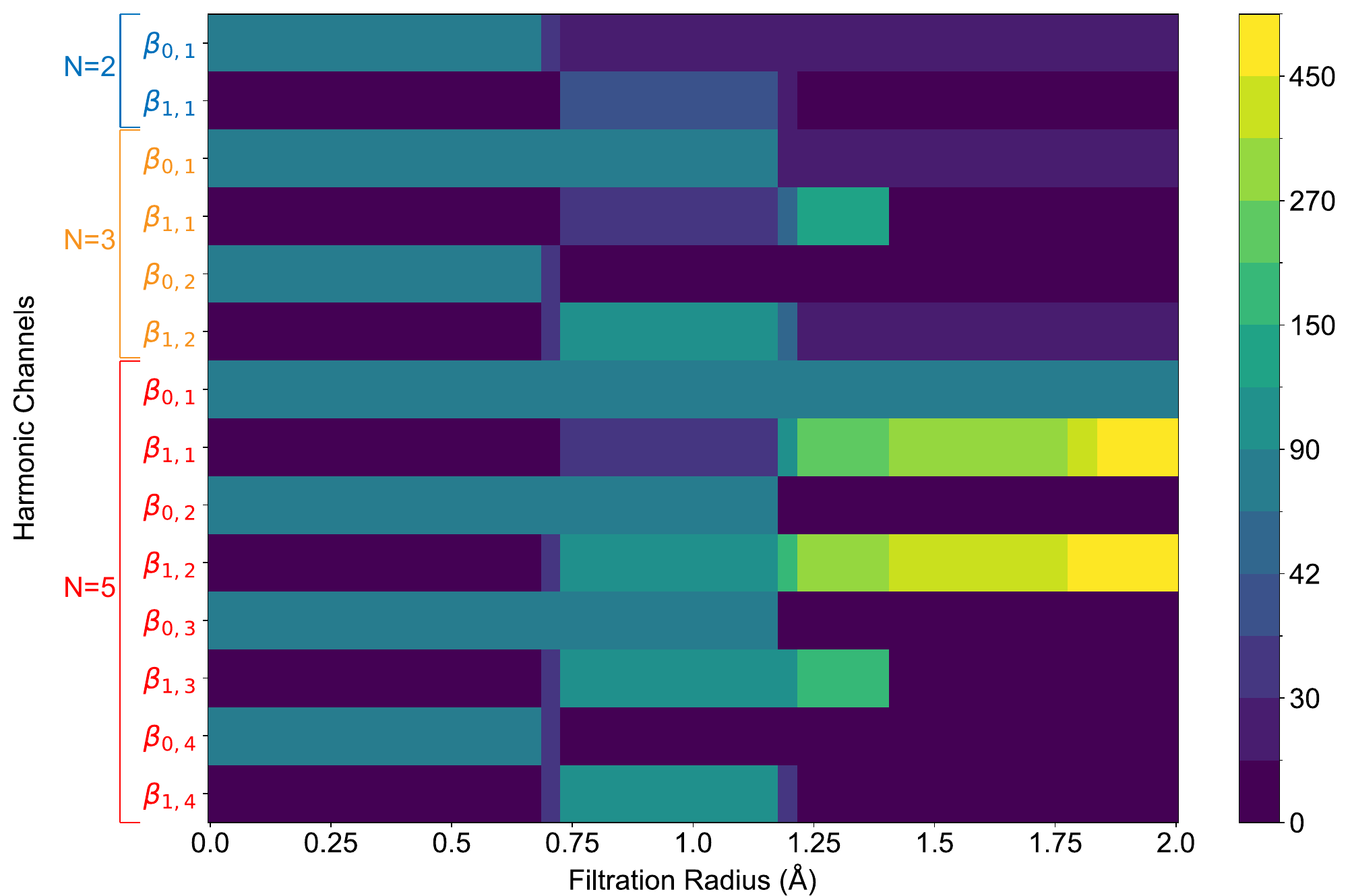}
	\caption{Comparison of persistent Betti numbers and the smallest positive eigenvalues of persistent Laplacians for fullerene $\mathrm{C}_{60}$ in cases where $N=2$, $N=3$, and $N=5$. Here, $\beta_{n,q}$ denotes the $n$-dimensional Betti number at stage $q$ for a given distance parameter. Similarly, $\lambda_{n,q}$ represents the smallest eigenvalue of the non-harmonic component of the Laplacian $\Delta_{n,q}$ at a given distance parameter.}\label{figure:C60_b}
\end{figure}

\begin{figure}[H]
	\centering
	\includegraphics[width=0.8\textwidth]{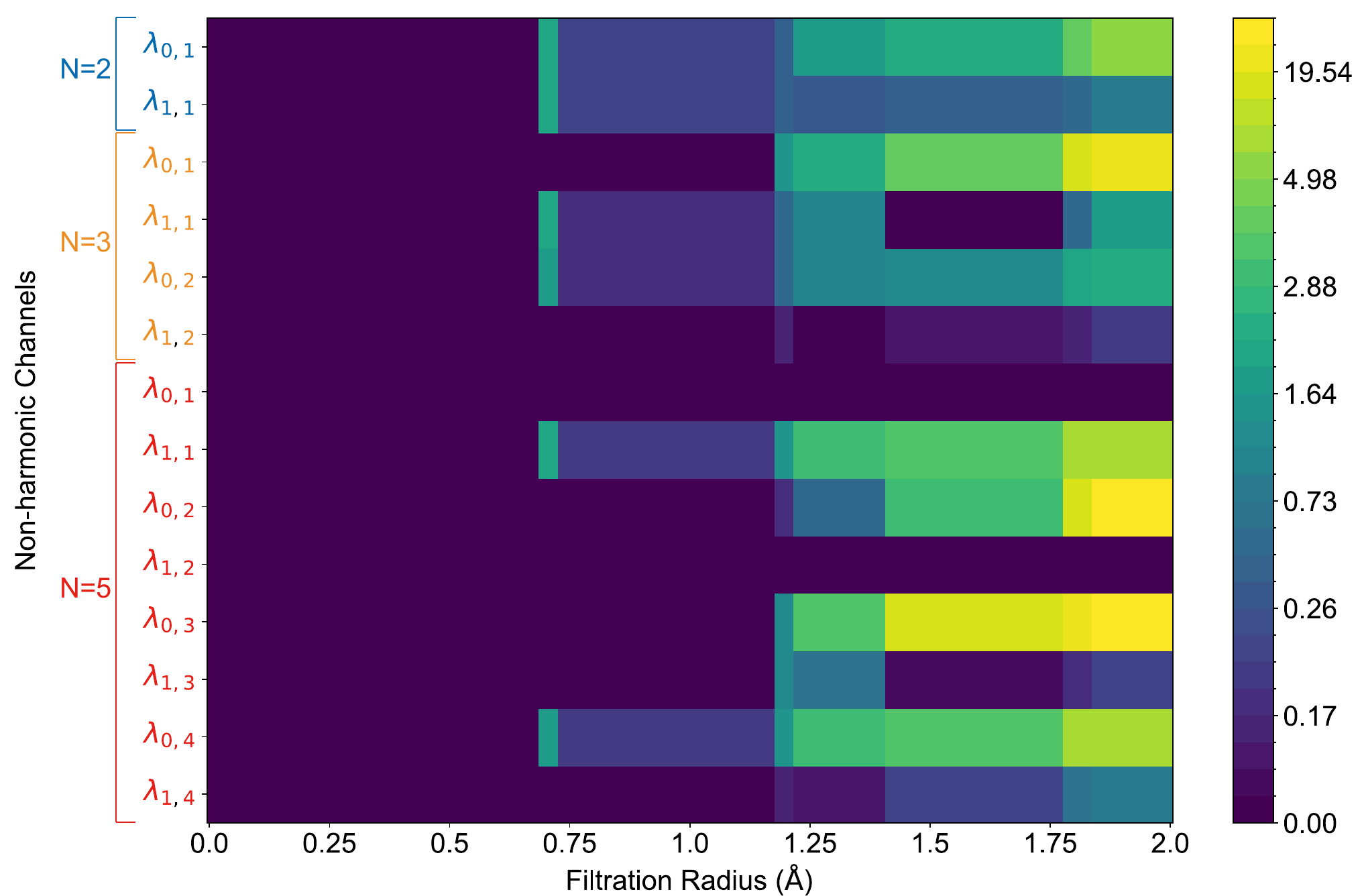}
	\caption{Comparison of persistent Betti numbers and the smallest positive eigenvalues of persistent Laplacians for fullerene $\mathrm{C}_{60}$ in cases where $N=2$, $N=3$, and $N=5$. Here, $\beta_{n,q}$ denotes the $n$-dimensional Betti number at stage $q$ for a given distance parameter. Similarly, $\lambda_{n,q}$ represents the smallest eigenvalue of the non-harmonic component of the Laplacian $\Delta_{n,q}$ at a given distance parameter.}\label{figure:C60_l}
\end{figure}

In the above calculations, for convenience, we computed the persistent Betti numbers and persistent spectral gaps of the 3-skeleton of the Vietoris-Rips complex. However, this does not hinder us from obtaining the topological and geometric characteristics of the structure. In the figures, we observe that for the case of $N=2$, the Betti numbers provide relatively limited information, while the spectral gaps can complement the geometric information.
For the cases of $N=3$ and $N=5$, the information contained in the Betti numbers alone is already comparable to the combined information of Betti numbers and spectral gaps for the $N=2$ case. This implies that, for larger values of $N$, computing Mayer Betti numbers alone is sufficient to capture the sum of harmonic and non-harmonic information present in the $N=2$ case. Generally, computing Betti numbers is much faster than solving for spectral gaps, providing a more efficient approach for calculating geometric features.

\begin{figure}[H]
	\centering
	\includegraphics[width=0.8\textwidth]{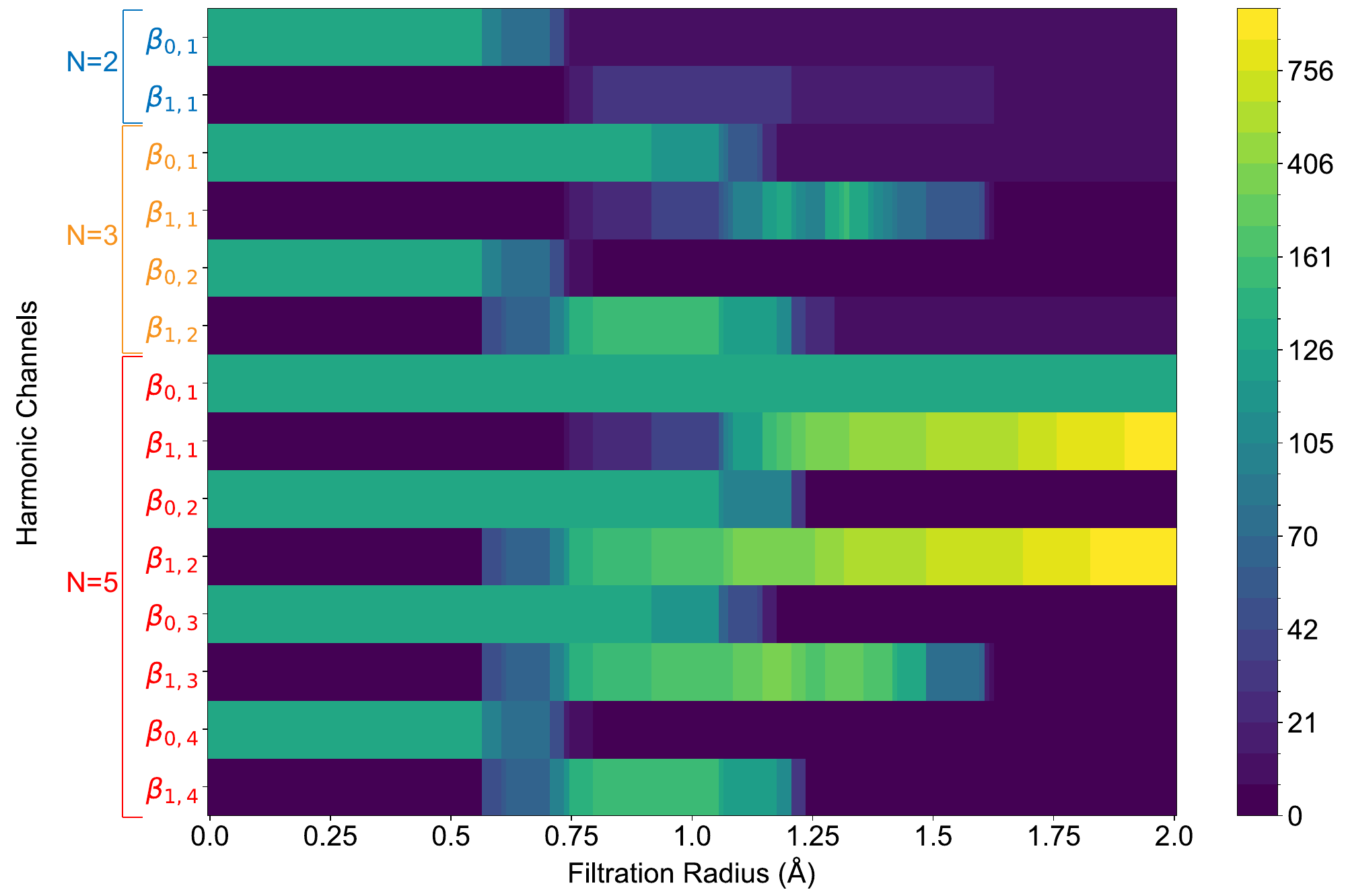}
	\caption{Comparison of persistent Betti numbers and the smallest positive eigenvalues of persistent Laplacians for cucurbit[7]uril $\mathrm{CB}7$ in cases where $N=2$, $N=3$, and $N=5$. Here, $\beta_{n,q}$ denotes the $n$-dimensional Betti number at stage $q$ for a given distance parameter. Similarly, $\lambda_{n,q}$ represents the smallest eigenvalue of the non-harmonic component of the Laplacian $\Delta_{n,q}$ at a given distance parameter.}\label{figure:CB7_b}
\end{figure}

\begin{figure}[H]
	\centering
	\includegraphics[width=0.8\textwidth]{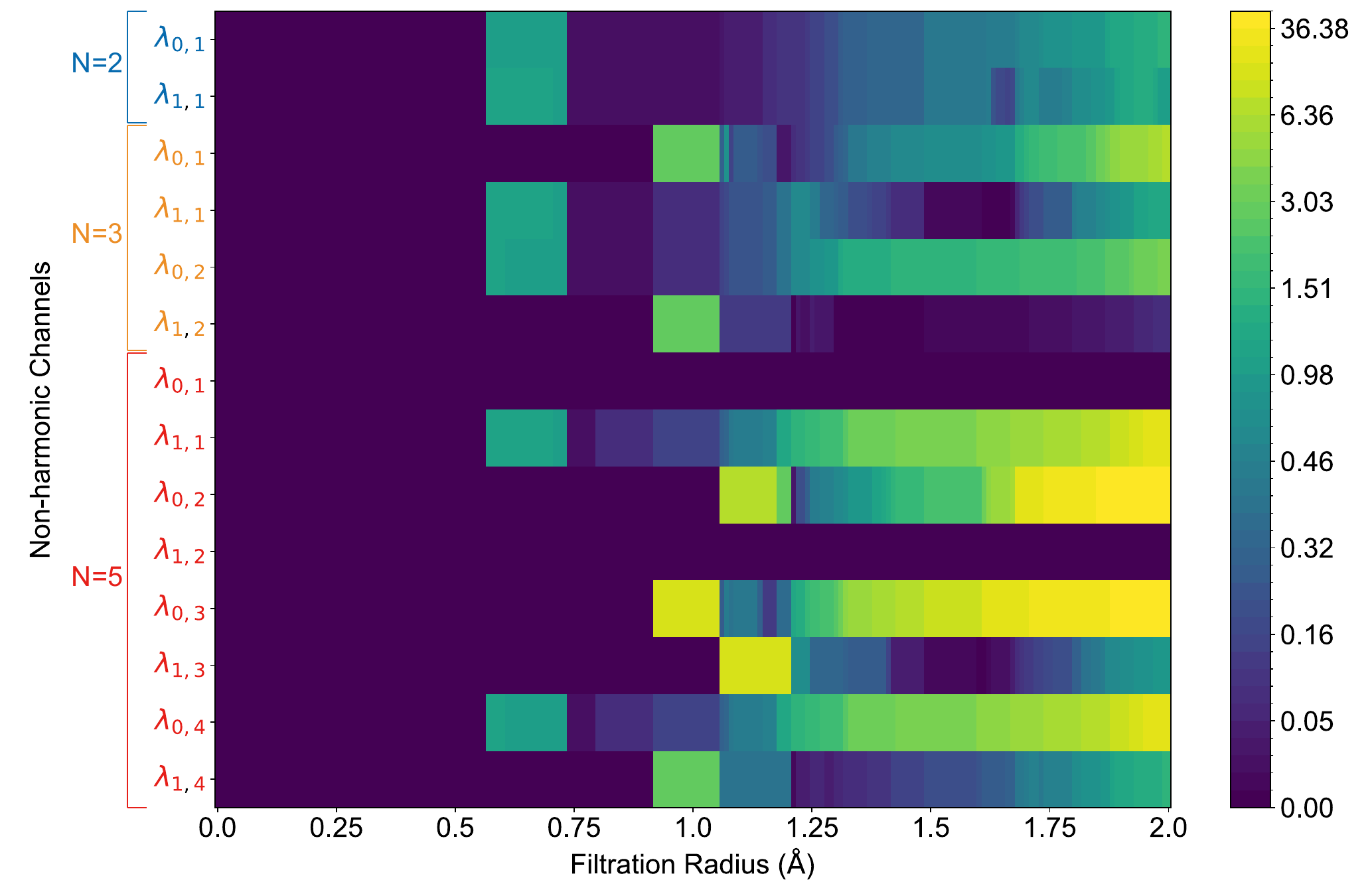}
	\caption{Comparison of persistent Betti numbers and the smallest positive eigenvalues of persistent Laplacians for cucurbit[7]uril $\mathrm{CB}7$ in cases where $N=2$, $N=3$, and $N=5$. Here, $\beta_{n,q}$ denotes the $n$-dimensional Betti number at stage $q$ for a given distance parameter. Similarly, $\lambda_{n,q}$ represents the smallest eigenvalue of the non-harmonic component of the Laplacian $\Delta_{n,q}$ at a given distance parameter.}\label{figure:CB7_l}
\end{figure}

Despite the calculation cost of persistent Mayer Laplacian, which should be approximately $N-1$ times that of the classical persistent Laplacian if we omit some of matrix multiplications, the persistent Mayer homology and persistent Mayer Laplacian, from an applied perspective, successfully provide practical multichannel featurization technique. As in applications, it is essential to obtain effective features of sufficient dimensionality before engaging in machine learning tasks, especially when dealing with datasets containing thousands or even millions of samples.

Traditional persistent homology and persistent Laplacian methods can only increase the feature dimensionality by adding more filtrations. This approach faces two main challenges. Firstly, there is an upper limit to the number of filtrations that can be added, and the computational cost becomes prohibitively high when dealing large filtration. Secondly, even with an increased number of filtrations, it does not guarantee the acquisition of useful information. This issue significantly impacts persistent homology, especially in higher dimensions (1-dimensional and above). In such scenarios, to obtain the desired features, it is common to divide the data into subgroups based on the physical understanding. For example, element-specific persistent homology considers different types of elements in the data \cite{cang2017topologynet}. Persistent Laplacians not only consider  the smallest positive eigenvalue but also take into account the largest eigenvalue and some statistical measures of the positive eigenvalues \cite{chen2022persistent}.

Persistent Mayer homology and persistent Mayer Laplacian possess Mayer degrees, serving as an additional dimension. By selecting specific values of $N$, we can effortlessly expand the feature dimensionality by a factor of $N-1$. Moreover, as the value of $N$ increases, each Mayer degree can have additional effective filtration choices for its corresponding features.  As shown in \autoref{figure:C60_b} and \autoref{figure:CB7_b}, more patterns in the persistent Mayer Betti numbers as $N$ increases.

\section{Conclusion}\label{section:conclusion}

To the best of our knowledge, currently known persistent homology and persistent Laplacians are constructed on chain complexes. For a long time, topological data analysis (TDA) relies on chain complexes to provide a framework for constructing all persistent homology and other persistent invariants. However, chain complex-based  persistent homology and persistent Laplacians are limited in their utility for dealing with real-world data challenges.

In this work, we consider $N$-chain complexes instead of chain complexes and construct persistent Mayer homology (PMH) and persistent Mayer Laplacians (PMLs) on $N$-chain complexes. Initially, we review some fundamental aspects of $N$-chain complexes, encompassing topics such as $N$-chain complexes, Mayer homology, and the construction of $N$-chain complexes on simplicial complexes. Additionally, we introduce Mayer Laplacians on $N$-chain complexes. We provide several computational examples of Mayer homology and Mayer Laplacians.
In our work, we consistently consider the simplicity of the case where $N$ is a prime number. However, the case in which $N$ is any integer greater than or equal to 2 can also be applicable for studying Mayer homology and Mayer Laplacians on $N$-chain complexes. In fact, our computational codes are applicable to all integer $N\geq2$.

The exploration of persistence on Mayer homology and Mayer Laplacians is pivotal in our work. We introduce persistent Mayer homology and explore the persistence diagram of persistent Mayer homology. Additionally, we investigate Wasserstein and bottleneck distances between Mayer persistence diagrams, establishing the stability of Mayer persistence diagrams. On the other hand, we introduce persistent Mayer Laplacians, providing additional geometric features to the spaces. On a discrete set of points in space, Vietoris-Rips complexes are obtained, allowing for the construction of $N$-chain complexes. Therefore, this work presents computations for the persistent Mayer homology and persistent Mayer Laplacians of finite point sets. The paper includes illustrative  figures and  examples.

Finally, we apply PMH and PMLs to small molecules, specifically the fullerene $\mathrm{C}_{60}$ and the cucurbit[7]uril $\mathrm{CB}7$.
Considering the coordinates of atoms in $\mathrm{C}_{60}$ and $\mathrm{CB}7$ as points in Euclidean space, we can obtain the corresponding Vietoris-Rips complexes. Subsequently, we compute the Mayer Betti numbers and spectral gaps of Mayer Laplacians for $N=2$, $N=3$, and $N=5$.
Additionally, for a given $N$, we calculate the Mayer Betti numbers and spectral gaps of Mayer Laplacians for $1\leq q\leq N-1$. These computational results unveil the distinctive properties of Mayer homology and Mayer Laplacian. In a broader context, Mayer homology stands as a homotopy invariant and, to some extent, offers a more powerful geometric characterization.

We believe that our approach gives rise to an emerging paradigm in TDA and offers fresh perspectives for  data science.  It will shed light on a wide range of undertakings, providing novel insights into real world problems. On the mathematical front,  persistent Mayer homology and persistent Mayer Laplacians can be further developed for various objects such as flag complex, path complex, directed graphs, hypergraphs, and hyperdigraphs. Additionally, persistent Mayer Dirac on various objects can be   formulated.
Theoretically,  $d^N=0$ gives rise to $N$-chain Mayer  Laplacian operators. One can use discrete Mayer Laplacians for data smoothing, image processing and many other applications.
Conceptually, Mayer Laplacians on manifolds may redefine heat equation, Schr\"{o}dinger equation,  Brownian motion, and conservation laws.
From an applied perspective, in the framework of topological deep learning
\cite{cang2017topologynet},  persistent Mayer homology and persistent Mayer Laplacians are expected to become powerful new data analysis tools for tackling data science challenges in diverse fields, including machine learning, physics, chemistry, biology, and materials science.

\section*{Data and Code Availability}
The data and source code obtained in this work are publicly available in the Github repository: \url{https://github.com/WeilabMSU/Persistence-Mayer-Homology-and-Laplacian}.

\section*{Acknowledgments}
This work was supported in part by NIH grants R01GM126189, R01AI164266, and R35GM148196, National Science Foundation grants DMS2052983, DMS-1761320, and IIS-1900473, NASA  grant 80NSSC21M0023,   Michigan State University Research Foundation, and  Bristol-Myers Squibb  65109.

\bibliographystyle{abbrv}
\bibliography{Reference}

\end{CJK*}
\end{document}